\documentclass[11pt,reqno]{amsart}    %字体大小
\usepackage{amsmath,amssymb,latexsym,soul,cite,mathrsfs,enumitem,empheq}
\usepackage{color,enumitem,graphicx}
\usepackage[colorlinks=true,urlcolor=blue,
citecolor=red,linkcolor=blue,linktocpage,pdfpagelabels,
bookmarksnumbered,bookmarksopen]{hyperref}
\usepackage[english]{babel}
\usepackage{enumitem}
\usepackage[left=3.0cm,right=3.0cm,top=2.8cm,bottom=2.8cm]{geometry}

\usepackage[hyperpageref]{backref}
\makeatletter
\newcommand{\leqnomode}{\tagsleft@true}
\newcommand{\reqnomode}{\tagsleft@false}
\makeatother

\numberwithin{equation}{section}

\newtheorem{thm}{Theorem}[section]
\newtheorem{lem}[thm]{Lemma}
\newtheorem{cor}[thm]{Corollary}
\newtheorem{Prop}[thm]{Proposition}

\DeclareMathOperator{\dist}{dist}

\DeclareMathOperator{\spn}{span}

\title[Biharmonic Br\'{e}zis-Nirenberg problem]{Energy asymptotics and blow-up phenomena for biharmonic Br\'{e}zis-Nirenberg problem}

\author[Jiamo Li, Qikai Lu, and Minbo Yang]{Jiamo Li, Qikai Lu*, and Minbo Yang}

\address{Jiamo Li  \newline\indent School of Mathematical Sciences, Zhejiang Normal University, \newline\indent Jinhua, Zhejiang, 321004, People's Republic of China}
\email{lijiamo@zjnu.edu.cn}

\address{Qikai Lu  \newline\indent School of Mathematical Sciences, Zhejiang Normal University, \newline\indent Jinhua, Zhejiang, 321004, People's Republic of China}
\email{luqikai@zjnu.edu.cn}

\address{Minbo Yang  \newline\indent School of Mathematical Sciences, Zhejiang Normal University, \newline\indent Jinhua, Zhejiang, 321004, People's Republic of China}
\email{mbyang@zjnu.edu.cn} 



\thanks{* Corresponding author}

\begin{document}
	\maketitle
	
	\begin{abstract}
		For dimensions $n\geq8$, we are concerned with the quotient functional of the biharmonic Br\'{e}zis-Nirenberg problem under the Navier boundary condition
		$$
		S(\varepsilon V):=\inf_{0\not\equiv u\in H^2(\Omega)\cap H_0^1(\Omega)}\frac{\int_{\Omega}|\Delta u|^2dx+\varepsilon\int_{\Omega}V|u|^2dx}{\big(\int_{\Omega}|u|^{2^\star}dx\big)^{2/2^\star}},
		$$
	    where $2^\star=\frac{2n}{n-4}$ is the critical Sobolev exponent of the embedding $H^2(\Omega)\cap H_0^1(\Omega)\hookrightarrow L^{2^\star}(\Omega)$, $\Omega\subset\mathbb{R}^n$ is a bounded open set and $V:\overline{\Omega}\rightarrow\mathbb{R}$ is a continuous function. Under certain assumptions on $V$, we establish sharp asymptotics for the energy difference $S(0)-S(\varepsilon V)$, as $\varepsilon\rightarrow0^+$, by means of matching upper and lower bound estimates. Moreover, we give a precise description of the blow-up profile of (almost) minimizing sequences and characterize the blow-up rate and the location of concentration points.
	\end{abstract}
	
	\section{Introduction And Main Results}
	In the celebrated paper \cite{BN1983}, Br\'{e}zis and Nirenberg studied the following quotient functional with Dirichlet boundary condition
	\begin{equation*}
		\widetilde{S}_{a}[u]:=\frac{\int_{\Omega}|\nabla u|^2dx+\int_{\Omega}a|u|^2dx}{\big(\int_{\Omega}|u|^{2n/(n-2)}dx\big)^{(n-2)/n}},\hspace{4mm}n\geq3,
	\end{equation*}
	over all $0\not\equiv u\in H_0^1(\Omega)$, where $\Omega\subset\mathbb{R}^n$ is a bounded open set and $a$ is a continuous function on $\overline{\Omega}$. They considered the attainment of the corresponding variational problem
	\begin{equation}\label{minimizer}
		\widetilde{S}(a):=\inf_{0\not\equiv u\in H_0^1(\Omega)}\widetilde{S}_{a}[u]
	\end{equation}
	and proved that, if $n\geq4$ the following properties are equivalent:
	\begin{enumerate}[label=(\roman*)]
		\item $\exists x_0\in\Omega$, s.t. $a(x_0)<0$,
		\item $\widetilde{S}(a)<\mathcal{S}_1$, where $\mathcal{S}_1$ is the best constant in the Sobolev inequality in $\mathbb{R}^n$,
		\item  $\widetilde{S}(a)$ is achieved by some positive function $u_a\in H_0^1(\Omega)$.
	\end{enumerate}
	For the case $n=3$, also called the critical dimensional case, the situation becomes more delicate. One of the findings of \cite{BN1983} is that if $a$ is small (for instance, in $L^\infty(\Omega)$) but possibly nonzero, then $\widetilde{S}(a)=\mathcal{S}_1$. This phenomenon naturally raises the question of how to characterize the strict inequality $\widetilde{S}(a)<\mathcal{S}_1$. More precisely, Br\'{e}zis \cite{Brezis1986} conjectured and Druet \cite{Druet2002} proved that, if $n=3$, the following equivalence holds:
	\begin{enumerate}[label=(\roman*)]
		\item $\exists x_0\in\Omega$, s.t. $\phi_a(x_0)<0$, where $\phi_a$ denotes the Robin function associated with the coercive operator $-\Delta+a$ in $H_0^1(\Omega)$,
		\item $\widetilde{S}(a)<\mathcal{S}_1$,
		\item $\widetilde{S}(a)$ is achieved by some positive function $u_a\in H_0^1(\Omega)$.
	\end{enumerate}
	Alternative proof of equivalence (i)$\Leftrightarrow$(ii) can be found in \cite{Esposito2004}. It is worth pointing out that the condition $\phi_a$ should be negative somewhere in $\Omega$ plays the same role as the condition $a(x)$ should be negative somewhere in $\Omega$ does in the case $n\geq4$.
	
	 The Br\'{e}zis-Nirenberg problem \eqref{minimizer} is also one of the classical models to understand the blow-up phenomena of nonlinear partial differential equations. From the variational point of view, Frank, K\"onig, and Kova\v{r}\'ik \cite{FKK2020,FKK2021} investigated the energy asymptotics of the minimization problem
	\begin{equation}\label{second order case}
		\widetilde{S}(a+\varepsilon V)=\inf_{0\not\equiv u\in H_0^1(\Omega)}\frac{\int_{\Omega}|\nabla u|^2dx+\int_{\Omega}(a+\varepsilon V)|u|^2dx}{\big(\int_{\Omega}|u|^{2n/(n-2)}dx\big)^{(n-2)/n}},
	\end{equation}
	where $V\in C(\overline{\Omega})$ and $a$ is assumed to be critical in the sense of Hebey and Vaugon \cite{HV2001}. More precisely, a continuous function $a$ on $\overline{\Omega}$ is said to be \textit{critical} if $\widetilde{S}(a)=\mathcal{S}_1$ and if for any continuous function $\hat{a}$ on $\overline{\Omega}$ with $\hat{a}\leq a$ and $\hat{a}\not\equiv a$ one has $\widetilde{S}(\hat{a})<\widetilde{S}(a)$. When $n\geq4$, it is easy to see that the only critical potential is $a\equiv0$. In contrast, for $n=3$, critical functions of arbitrary shape exist (see \cite{Druet2002}). In \cite{FKK2020,FKK2021}, the authors provided a precise description of the blow-up profile of (almost) minimizing sequences and characterized the location of the concentration points in dimensions $n\geq4$ and in the critical dimension $n=3$, respectively. It is worth noting that these two regimes exhibit a fundamental distinction. As shown in \cite{FKK2020}, the asymptotic coefficient depends pointwise on the potential $V$, whereas in \cite{FKK2021} this dependence becomes non-local in dimension three. This finding is consistent with the Br\'{e}zis-Nirenberg phenomenon mentioned above. Also, we observe that analogous energy asymptotics in the case where $V$ is a negative constant and $n\geq4$ are obtained in \cite{Takahashi2004}. From the perspective of Euler–Lagrange equations involving one-bubble solutions, Br\'{e}zis and Peletier \cite{BP1989} considered the case where $\Omega$ is a ball and $a$ and $V$ are constants. For general open sets $\Omega$, they formulated three conjectures. The first two were resolved independently by Rey \cite{Rey1989} and Han \cite{Han1991}, who established the blow-up asymptotics of positive minimizers for problem \eqref{second order case} in dimensions $n\geq4$ with $a\equiv0$ and for a related problem involving nearly critical nonlinearity in dimensions $n\geq3$ with $a\equiv0$. Very recently, Frank, K\"onig, and Kova\v{r}\'ik \cite{FKK2024} proved the third conjecture which concerns certain nonzero $a$ in dimension three. Additionally, there has been considerable interest in the multibubble blow-up analysis. See for example, \cite{CLP2021,KL2022,KL2024,MP2002} and the references therein.
	
	Let $\Omega\subset\mathbb{R}^n$ be a bounded open set and $V:\overline{\Omega}\rightarrow\mathbb{R}$ be a continuous function. Inspired by Frank, K\"onig, and Kova\v{r}\'ik \cite{FKK2020,FKK2021}, we consider the following minimization problem
		\begin{equation}\label{wenti}
			S(\varepsilon V):=\inf_{0\not\equiv u\in H^2(\Omega)\cap H_0^1(\Omega)}S_{\varepsilon V}[u],
		\end{equation}
	where 
	\begin{equation}\label{mianbao}
		S_{\varepsilon V}[u]:=\frac{\int_{\Omega}|\Delta u|^2dx+\varepsilon\int_{\Omega}V|u|^2dx}{\big(\int_{\Omega}|u|^{2^\star}dx\big)^{2/2^\star}},\hspace{4mm}n\geq5,
	\end{equation}
	and $2^\star=\frac{2n}{n-4}$ is the critical Sobolev exponent of the embedding $\mathcal{H}:=H^2(\Omega)\cap H_0^1(\Omega)\hookrightarrow L^{2^\star}(\Omega)$. This number is to be compared with 
	\begin{equation}\label{best constant}
		\mathcal{S}_2=\pi^2n(n-4)(n^2-4)\Gamma(n/2)^\frac{4}{n}\Gamma(n)^{-\frac{4}{n}},
	\end{equation}
	the best constant \cite{EFJ1990,Lieb1983} in the Sobolev inequality
	\begin{equation}\label{Sobolev inequality}
		\mathcal{S}_2\bigg(\int_{\mathbb{R}^n}|u|^{2^\star}dx\bigg)^\frac{2}{2^\star}\leq\int_{\mathbb{R}^n}|\Delta u|^2dx,\hspace{6mm}u\in D^{2,2}(\mathbb{R}^n),
	\end{equation}
	where $D^{2,2}(\mathbb{R}^n)$ is the completion of $C_0^\infty(\mathbb{R}^n)$ under the norm
	\begin{equation*}
		\|u\|_{D^{2,2}(\mathbb{R}^n)}=\bigg(\int_{\mathbb{R}^n}|\Delta u|^2dx\bigg)^\frac{1}{2}.
	\end{equation*}
	Let $\mu_1$ denote the first eigenvalue of the biharmonic operator $\Delta^2$ on $\Omega$ under the Navier boundary condition, i.e. $u=\Delta u=0$ on $\partial\Omega$. The following result (assuming $V\equiv-1$), due to Van der Vorst \cite{V1995}, extends the seminal Br\'{e}zis-Nirenberg's result to the fourth order critical elliptic equation:
	\begin{enumerate}
		\item If $n\geq8$, then for every $\varepsilon\in(0,\mu_1)$ we have $S(-\varepsilon)<\mathcal{S}_2$.
		\item If $5\leq n\leq7$, then there exists a number $\mu_*>0$ such that for every $\varepsilon\in(\mu_*,\mu_1)$, we have $S(-\varepsilon)<\mathcal{S}_2$.
	\end{enumerate}			
	 If $\varepsilon\equiv0$, the situation becomes much more subtle and\cite{V1995} also showed that if $\Omega$ is star--shaped, the corresponding Euler--Lagrange equation has no positive solution. Ebobisse and Ahmedou \cite{EA2003} investigated the effect of domain's topology on the existence of solutions. They proved that the equation admits a positive solution provided that some homology group of $\Omega$ is nontrivial, which generalizes the famous result of Bahri--Coron \cite{BC1988} initially obtained for the second order elliptic problem. This topological condition is sufficient, but not necessary, because Gazzola, Grunau, and Squassina \cite{GGS2003} gave an example of contractible domains on which a positive solution still exists.
	
	The results mentioned above indicate that it is interesting to study the asymptotic behavior of the minimizers for the biharmonic problem \eqref{wenti} as $\varepsilon\rightarrow0^+$. In the special case where $V$ is a negative constant, El Mehdi and Selmi \cite{ElS2006} constructed positive blow-up solutions to the corresponding Euler-Lagrange equation, under the assumption that $u_\varepsilon$ is a minimizer of \eqref{wenti} in dimensions $n\geq9$, and showed that these solutions concentrate around critical points of the Robin function $H(x,x)$ (see \eqref{Robin function} below). Conversely, they proved that any nondegenerate critical point of $H(x,x)$ generates a family of positive minimizers. Such concentration phenomena extend to the fourth order equations some results obtained by Rey \cite{Rey1989,Rey1990} for the second order elliptic equations.  For other blow-up results involving the biharmonic operator, we refer the interested readers to \cite{CG2000,ElH2005,Geng1999,Geng2005,Wei1996} and the references therein.
	
	The aim of this paper is to determine the asymptotics of the perturbed minimal energy \eqref{wenti} and to understand the behavior of corresponding minimizers as $\varepsilon\rightarrow0^+$, without imposing any sign condition on the minimizers. Before stating our main results, we need to introduce some notations. Let us define on $\Omega$ the Robin function $R(x):=H(x,x)$ with
	\begin{equation}\label{Robin function}
		H(x,y)=|x-y|^{4-n}-G(x,y),\hspace{6mm}\text{for}\hspace{2mm}(x,y)\in\Omega\times\Omega,
	\end{equation}
	where $G$ is the Green's function of $\Delta^2$ under the Navier boundary condition, that is,
	\begin{equation*}
		\left\lbrace
		\begin{aligned}
			\Delta^2G(\cdot,y)=c_0\delta_y&\hspace{9mm}\text{in}\hspace{2mm}\Omega,\\
			G(\cdot,y)=\Delta G(\cdot,y)=0&\hspace{9mm}\text{on}\hspace{2mm}\partial\Omega,
		\end{aligned}
		\right.
	\end{equation*}
	where $c_0=(n-4)(n-2)|\mathbb{S}^{n-1}|$ and $\delta_y$ denotes the Dirac mass at $y$.
	
	For $\lambda>0$ and $x\in\mathbb{R}^n$, let
	\begin{equation}\label{A1.3}
		U_{x,\lambda}(y)=c_n\bigg(\frac{\lambda}{1+\lambda^2|y-x|^2}\bigg)^\frac{n-4}{2},\hspace{4mm}\text{with}\hspace{2mm}c_n=[(n-4)(n-2)n(n+2)]^\frac{n-4}{8}.
	\end{equation}
	It is well-known (see \cite{Lin1998,WX1999}) that $U_{x,\lambda}$ are the only solutions of
	\begin{equation*}
		\Delta^2u=u^\frac{n+4}{n-4},\hspace{4mm}u>0\hspace{2mm}\text{in}\hspace{2mm}\mathbb{R}^n,
	\end{equation*}
	with $u\in L^{2^\star}(\mathbb{R}^n)$ and $\Delta u\in L^2(\mathbb{R}^n)$, and are also the only minimizers of the Sobolev inequality \eqref{Sobolev inequality}. We denote by $PU_{x,\lambda}$ the projection of $U_{x,\lambda}$ on $\mathcal{H}$ satisfying
	\begin{equation*}
		\left\lbrace
		\begin{aligned}
			\Delta^2PU_{x,\lambda}=\Delta^2U_{x,\lambda}&\hspace{9mm}\text{in}\hspace{2mm}\Omega,\\
			PU_{x,\lambda}=\Delta PU_{x,\lambda}=0&\hspace{9mm}\text{on}\hspace{2mm}\partial\Omega.
		\end{aligned}
		\right.
	\end{equation*}
	The space $\mathcal{H}$ is equipped with the norm $\|\cdot\|$ and its corresponding inner product $(\cdot,\cdot)$ defined by
	\begin{equation*}
		\|u\|=\bigg(\int_{\Omega}|\Delta u|^2dx\bigg)^\frac{1}{2},\hspace{4mm}u\in\mathcal{H},
	\end{equation*}
	\begin{equation*}
		(u,v)=\int_{\Omega}\Delta u\Delta vdx\hspace{6mm}\text{for}\hspace{2mm}u,v\in \mathcal{H}.
	\end{equation*}
    We let $|\cdot|_{L^p(\Omega)}$ denote the usual $L^p$ norm in $\Omega$. Moreover, let
    \begin{equation*}
    	T_{x,\lambda}:=\spn\bigg\{PU_{x,\lambda},\partial_\lambda PU_{x,\lambda},\partial_{x_1}PU_{x,\lambda},\cdots,\partial_{x_n}PU_{x,\lambda}\bigg\}
    \end{equation*}
    and let $T_{x,\lambda}^\perp$ be the orthogonal complement of $T_{x,\lambda}$ in $\mathcal{H}$ with respect to the inner product. We introduce the quantity
    \begin{equation*}
    	\Phi_n:=
    	\left\lbrace
    	\begin{aligned}
    		\sup\limits_{x\in\mathcal{N}(V)}R(x)^{-1}|V(x)|&\hspace{9mm}\text{if}\hspace{2mm}n=8,\\
    		\sup\limits_{x\in\mathcal{N}(V)}R(x)^{-\frac{4}{n-8}}|V(x)|^{\frac{n-4}{n-8}}&\hspace{9mm}\text{if}\hspace{2mm}n\geq9,
    	\end{aligned}
    	\right.
    \end{equation*}
    where $\mathcal{N}(V):=\{x\in\Omega, V(x)<0\}$. Furthermore, set
    \begin{equation*}
    	\mathfrak{a}_n=\int_{\mathbb{R}^n}\frac{\mathrm{d}z}{(1+z^2)^{\frac{n+4}{2}}}
    	=\frac{2\omega_n}{n(n+2)},\hspace{6mm}n\geq8,
    \end{equation*}
    \begin{equation*}
    	\mathfrak {b}_n=\left\lbrace
    	\begin{aligned}
    		\omega_8&\hspace{9mm}\text{if}\hspace{2mm}n=8\\
    		\int_{\mathbb{R}^n}\frac{\mathrm{d}z}{(1+z^2)^{n-4}}&\hspace{9mm}\text{if}\hspace{2mm}n\geq9
    	\end{aligned}
    	\right.
    	=\left\lbrace\begin{aligned}
    		\omega_8&\hspace{9mm}\text{if}\hspace{2mm}n=8,\\
    		\omega_n\frac{\Gamma(\frac{n}{2}-4)\Gamma(\frac{n}{2})}{2\Gamma(n-4)}&\hspace{9mm}\text{if}\hspace{2mm}n\geq9,
    	\end{aligned}
    	\right.
    \end{equation*}
    \begin{equation*}
    	\mathfrak{C}_n=\frac{n-8}{n-4}\bigg(\frac{4}{n-4}\bigg)^{\frac{4}{n-8}}\mathfrak {a}_n^{-\frac{4}{n-8}}\mathfrak {b}_n^{\frac{n-4}{n-8}}c_n^{\frac{2n}{n-4}-\frac{8}{n-8}}\mathcal{S}_2^{\frac{4-n}{4}},\hspace{6mm}n\geq9,
    \end{equation*}
    
    \begin{equation*}
    	\mathfrak{D}_n=\bigg(\frac{4}{n-4}\bigg)^{\frac{n-4}{n-8}}\mathfrak {a}_n^{-\frac{4}{n-8}}\mathfrak {b}_n^{\frac{n-4}{n-8}}c_n^{\frac{2n}{n-4}-\frac{8}{n-8}}\mathcal{S}_2^{-\frac{n}{4}},\hspace{6mm}n\geq9,
    \end{equation*}
    where $\omega_n$ is the volume of the $n-1$ dimensional unit sphere in $\mathbb{R}^n$.
	
	Our first main result can be stated as follows.
	\begin{thm}\label{thm1}
		Assume $n\geq8$, $\Omega\subset\mathbb{R}^n$ is a bounded open set with $C^2$ boundary. Assume $V:\overline{\Omega}\rightarrow\mathbb{R}$ is a continuous function with $\mathcal{N}(V)\neq\emptyset$. As $\varepsilon\to0^+$, we have
		\begin{equation*}
			\mathcal{S}_2-S(\varepsilon V)=\left\lbrace
			\begin{aligned}
				e^{-\frac{c_8^2 }{10\Phi_8\varepsilon}(1+o(1))}\hspace{9mm}\text{if}\hspace{2mm}n=8,\\
				\mathfrak{C}_n\Phi_n\varepsilon^{\frac{n-4}{n-8}}+o(\varepsilon^{\frac{n-4}{n-8}})\hspace{9mm}\text{if}\hspace{2mm}n\geq 9,
			\end{aligned}
			\right.
		\end{equation*}
		where the constant $\mathcal{S}_2$ is defined in \eqref{best constant}.
	\end{thm}
	
	Our second main result shows that the blow-up profile of an arbitrary almost minimizing sequence $\{u_\varepsilon\}$ is given to leading order by the family of functions $PU_{x,\lambda}$. Moreover, we provide a precise characterization of both the blow-up rate and the concentration point.
	
	\begin{thm}\label{thm2}
		Assume $n\geq8$, $\Omega\subset\mathbb{R}^n$ is a bounded open set with $C^2$ boundary. Assume $V:\overline{\Omega}\rightarrow\mathbb{R}$ is a continuous function with $\mathcal{N}(V)\neq\emptyset$. Let $\{u_\varepsilon\}\subset \mathcal{H}$ be a family of functions satisfying 
		\begin{equation}\label{A1.1}
			\lim_{\varepsilon\to 0}\frac{S_{\varepsilon V}[u_\varepsilon]-S(\varepsilon V)}{\mathcal{S}_2-S(\varepsilon V)}=0\quad\text{and}\quad\int_{\Omega}|u_\varepsilon|^{2^\star}\,\mathrm{d}x=\mathcal{S}_2^{\frac{n}{4}}.
		\end{equation}
		Then there exist $\{x_\varepsilon\}\subset\Omega$, $\{\lambda_\varepsilon\}\subset(0,\infty)$, $\{\alpha_\varepsilon\}\subset\mathbb{R}$ and $\{v_\varepsilon\}\subset\mathcal{H}$ with $v_\varepsilon\in T_{x_\varepsilon,\lambda_\varepsilon}^\perp$ such that
		$$u_\varepsilon=\alpha_\varepsilon(PU_{x_\varepsilon,\lambda_\varepsilon}+v_\varepsilon).$$
		Up to a subsequence, we have $x_\varepsilon\to x_0$ for some $x_0\in\mathcal{N}(V)$. Moreover,
		\begin{equation}\label{further locate}
			\left\lbrace
			\begin{aligned}
				R(x_0)^{-1}|V(x_0)|=\Phi_8&\hspace{9mm}\text{if}\hspace{2mm}n=8,\\
				R(x_0)^{-\frac{4}{n-8}}|V(x_0)|^{\frac{n-4}{n-8}}=\Phi_n&\hspace{9mm}\text{if}\hspace{2mm}n\geq9,
			\end{aligned}
			\right.
		\end{equation}
		
		\begin{equation*}
			\left\lbrace
			\begin{aligned}
				\|v_\varepsilon\|\leq e^{-\frac{c_8^2}{20\Phi_8\varepsilon}(1+o(1))}&\hspace{9mm}\text{if}\hspace{2mm}n=8,\\
				\|v_\varepsilon\|=o(\varepsilon^{\frac{n-4}{2n-16}})&\hspace{9mm}\text{if}\hspace{2mm}n\geq9,
			\end{aligned}
			\right.
		\end{equation*}
		
		\begin{equation*}
			\left\lbrace
			\begin{aligned}
				\lim_{\varepsilon\rightarrow0}\varepsilon\log\lambda_\varepsilon=\frac{c_8^2R(x_0)}{40|V(x_0)|}&\hspace{9mm}\text{if}\hspace{2mm}n=8,\\
				\lim_{\varepsilon\rightarrow0}\varepsilon\lambda_\varepsilon^{n-8}=\frac{(n-4)\mathfrak {a}_nc_n^{2^\star-2}R(x_0)}{4\mathfrak {b}_n|V(x_0)|}&\hspace{9mm}\text{if}\hspace{2mm}n\geq9,
			\end{aligned}
			\right.
		\end{equation*}
		
		\begin{equation*}
			\left\lbrace
			\begin{aligned}
				|\alpha_\varepsilon|=1+e^{-\frac{c_8^2}{10\Phi_8\varepsilon}(1+o(1))}&\hspace{9mm}\text{if}\hspace{2mm}n=8,\\
				|\alpha_\varepsilon|=1+\mathfrak{D}_n\Phi_n\varepsilon^{\frac{n-4}{n-8}}+o(\varepsilon^{\frac{n-4}{n-8}})&\hspace{9mm}\text{if}\hspace{2mm}n\geq9.
			\end{aligned}
			\right.
		\end{equation*}
	\end{thm}
	
    We now briefly comment on our main theorems. The energy asymptotics established in Theorem \ref{thm1} appear to be new for the fourth order Br\'{e}zis-Nirenberg problem under the Navier boundary condition. We emphasize that Theorem \ref{thm2} not only extends the results of \cite{ElS2006} to dimension $n=8$, but also provides a more precise description of the blow-up points as being extrema of a quotient involving the Robin's function (see \eqref{further locate}). Furthermore, we treat directly with the variational problem \eqref{wenti} and not with the Euler-Lagrange equation. Therefore, our blow-up results hold not only for minimizers but even for almost minimizers in the sense of \eqref{A1.1}. On the other hand, compared to \cite{ElS2006}, a limitation of our approach is that blow-up estimates are established only in the $H^2$-norm rather than the $L^\infty$-norm. Finally, we remark that our analysis is restricted to linear perturbations, namely the exponent of the perturbation term in \eqref{mianbao} is fixed to be two, rather than considering general subcritical perturbations. The reason is that the proof of Proposition \ref{homogeneity} relies crucially on the homogeneity of the associated quotient functional.
    
    The assumption $n\geq8$ is sharp, since for $5\leq n\leq7$, minimization problem \eqref{wenti} has no minimizer when $\Omega$ is a ball for $\varepsilon$ small \cite[Theorem 3]{GGS2003}. The energy asymptotics and concentration results are much more difficult to obtain in these low dimensions. In sharp contrast to the second order critical elliptic equations, very little is known about the fourth order case in the critical dimensions $5\leq n\leq7$. To the best of our knowledge, there has no characterization of criticality of the function $a$ in low dimensions for a general domain $\Omega$. We notice, however, that the complete characterization of critical functions $a$ in low dimensions has recently been established in the fractional setting \cite{DeK2023} and in the quasilinear framework \cite{AE2024}.
    
    Throughout the paper we denote by $c,c_0,C,C_1,C_2,\cdots$ various positive constants which may vary from line to line and are not essential to the problem. We shall write that $a\lesssim b$ (resp. $a\gtrsim b$) if $a\leq Cb$ (resp. $Ca\geq b$). An expression $A(\varepsilon)$ depending on $\varepsilon>0$ will said to be $O(\varepsilon)$ (resp. $o(\varepsilon)$) if, all other parameters being fixed, $|A(\varepsilon)/\varepsilon|\leq C$ (resp. $A(\varepsilon)/\varepsilon\rightarrow0$).
    
    The paper is organized as follows. In Section 2 we establish the upper bound estimate by employing the family of functions $PU_{x,\lambda}$ as test functions. Section 3 is devoted to the lower bound estimate, where we derive an asymptotic decomposition for a general almost minimizing sequence $\{u_\varepsilon\}$ as well as the corresponding expansion of $S_{\varepsilon V}[u_\varepsilon]$. Then in Section 4 we show the proof of Theorems \ref{thm1} and \ref{thm2}. Finally, an appendix contains two auxiliary technical results.

\section{Upper Bound Estimate}
In this section, we concentrate on the upper bound estimate for the minimal energy $S(\varepsilon V)$. To this end, we employ the family of test functions $PU_{x,\lambda}$, where the parameters $x\in\Omega$ and $\lambda>0$ are chosen suitably, and derive a precise expansion of the value $S_{\varepsilon V}[PU_{x,\lambda}]$. Let $$d(x)=\dist(x,\partial\Omega)$$
 denote the distance from a point $x\in\Omega$ to the boundary $\partial\Omega$.
\begin{thm}\label{thm2.1}
	Let $x=x_\lambda$ be a sequence of points such that $d(x)\lambda\to+\infty$. Then as $\lambda\to+\infty$, we have 
	\begin{equation}\label{A1.10}
		\int_{\Omega}|\Delta PU_{x,\lambda}|^2\,\mathrm{d}y=\mathcal{S}_2^{\frac{n}{4}}-\mathfrak {a}_nc_n^{2^\star}R(x)\lambda^{4-n}+O\big((d(x)\lambda)^{\frac{16}{5}-n}\big),
	\end{equation}
	\begin{equation}\label{A1.10+}
		\int_{\Omega}V|PU_{x,\lambda}|^2\,\mathrm{d}y=
		\left\lbrace
		\begin{aligned}
			\mathfrak {b}_8c_8^2V(x)\lambda^{-4}\log\lambda+O\big((d(x)\lambda)^{-4}\big)+o(\lambda^{-4}\log\lambda)&\hspace{9mm}\text{if}\hspace{2mm}n=8,\\
			\mathfrak {b}_nc_n^2V(x)\lambda^{-4}+O\big((d(x)\lambda\big)^{4-n})+o(\lambda^{-4})&\hspace{9mm}\text{if}\hspace{2mm}n\geq9,
		\end{aligned}
		\right.
	\end{equation}
	and
	\begin{equation}\label{A1.10++}
		\int_{\Omega}| PU_{x,\lambda}|^{2^\star}\,\mathrm{d}y=\mathcal{S}_2^{\frac{n}{4}}-2^\star\mathfrak {a}_n c_n^{2^\star}R(x)\lambda^{4-n}+o\big((d(x)\lambda\big)^{4-n}).
	\end{equation}
	In particular, we have
	\begin{equation}\label{A1.10+++}
		S_{\varepsilon V}[PU_{x,\lambda}]=
		\left\lbrace
		\begin{aligned}
			\mathcal{S}_2+\frac{\mathfrak {a}_8c_8^{4}R(x)+\varepsilon \mathfrak {b}_8c_8^2V(x)\log\lambda}{\mathcal{S}_2\lambda^{4}}+o\big((d(x)\lambda)^{-4}\big)+o(\varepsilon\lambda^{-4}\log\lambda)&\hspace{9mm}\text{if}\hspace{2mm}n=8,\\
			\mathcal{S}_2+\mathcal{S}_2^{1-\frac{n}{4}}\bigg(\frac{\mathfrak {a}_nc_n^{2^\star}R(x)}{{\lambda^{n-4}}}+\frac{\varepsilon \mathfrak {b}_nc_n^2V(x)}{\lambda^{4}}\bigg)+o\big((d(x)\lambda\big)^{4-n})+o(\varepsilon\lambda^{-4})&\hspace{9mm}\text{if}\hspace{2mm}n\geq9,
		\end{aligned}
		\right.
	\end{equation}
	as $\lambda\to+\infty$.
\end{thm}
\begin{proof}
	We prove the estimates \eqref{A1.10}–\eqref{A1.10++} and expansion \eqref{A1.10+++} follows immediately by applying the Taylor expansion to the quotient functional $S_{\varepsilon V}[PU_{x,\lambda}]$. A straightforward integration by parts shows that
	\begin{equation}\label{A1.12}
		\int_{\Omega}|\Delta PU_{x,\lambda}|^2\,\mathrm{d}y=\int_{\Omega}U_{x,\lambda}^{\frac{n+4}{n-4}}PU_{x,\lambda}\,\mathrm{d}y.
	\end{equation} 
	Moreover, Proposition 2.1 in \cite{AH2006} provides some necessary estimates
	\begin{equation}\label{A1.13}
		U_{x,\lambda}=PU_{x,\lambda}+\theta_{x,\lambda},\quad\theta_{x,\lambda}=c_n\lambda^{\frac{4-n}{2}}H(x,y)+f_{x,\lambda},
	\end{equation}
	where $f_{x,\lambda}$ satisfies
	\begin{equation}\label{A1.14}
		f_{x,\lambda}=O\bigg(\frac{1}{\lambda^{\frac{n}{2}}d^{n-2}}\bigg).
	\end{equation}
	Combining \eqref{A1.12} and \eqref{A1.13}, it follows that
	\begin{equation}\label{A1.15}
		\int_{\Omega}|\Delta PU_{x,\lambda}|^2\,\mathrm{d}y=\int_{\Omega}U_{x,\lambda}^{2^\star}\,\mathrm{d}y-c_n\lambda^{\frac{4-n}{2}}\int_{\Omega}U_{x,\lambda}^{\frac{n+4}{n-4}}H(x,y)\,\mathrm{d}y-\int_{\Omega}U_{x,\lambda}^{\frac{n+4}{n-4}}f_{x,\lambda}\,\mathrm{d}y.
	\end{equation}
	By virtue of the definition of $\mathcal{S}_2$, together with \eqref{A1.3}, we derive
	\begin{equation}\label{A1.16}
		\int_{\Omega}U_{x,\lambda}^{2^\star}\,\mathrm{d}y=\int_{\mathbb{R}^n}U_{x,\lambda}^{2^\star}\,\mathrm{d}y+O\big((d(x)\lambda\big)^{-n})=\mathcal{S}_2^{\frac{n}{4}}+O\big((d(x)\lambda)^{-n}\big).
	\end{equation}
	Moreover, for all $x\in\Omega$,
	\begin{equation}\label{A1.17}
		d(x)^{4-n}\lesssim H(x,y)\lesssim d(x)^{4-n}\hspace{4mm}\text{and}\hspace{4mm}|\nabla_y H(x,y)|\lesssim d(x)^{3-n},
	\end{equation}
	where the first estimate follows from an iterated application of the maximum principle. Invoking \eqref{A1.3} and \eqref{A1.17}, we get
	\begin{equation*}
		\begin{aligned}
			\int_{B_{\rho}(x)}U_{x,\lambda}^{\frac{n+4}{n-4}}H(x,y)\,\mathrm{d}y
			=&c_n^{\frac{n+4}{n-4}}\lambda^{\frac{n+4}{2}}(R(x)+O(\rho d(x)^{3-n}))\int_{B_{\rho}(x)}\frac{\mathrm{d}y}{(1+\lambda^2|x-y|^2)^{\frac{n+4}{2}}}\\
			=&\mathfrak {a}_nc_n^{\frac{n+4}{n-4}}\lambda^{\frac{4-n}{2}}\bigg(R(x)+O\big(\rho d(x)^{3-n}\big)\bigg)\big(1+O((\lambda \rho)^{-4})\big)
		\end{aligned}
	\end{equation*}
	and
	\begin{equation*}
		\begin{aligned}
			\int_{\Omega\setminus B_{\rho}(x)}U_{x,\lambda}^{\frac{n+4}{n-4}}H(x,y)\,\mathrm{d}y
			=&c_n^{\frac{n+4}{n-4}}\lambda^{\frac{n+4}{2}}O\big( d(x)^{4-n}\big)\int_{\rho}^{+\infty}\frac{r^{n-1}}{(1+(\lambda r)^2)^{\frac{n+4}{2}}}\,\mathrm{d}r\\
			=&c_n^{\frac{n+4}{n-4}}\lambda^{\frac{4-n}{2}}O\big( d(x)^{4-n}(\lambda\rho)^{-4}\big),
		\end{aligned}
	\end{equation*}
	where $0<\rho<d(x)/2$. Hence, it can be inferred that
	\begin{equation*}
			\lambda^{\frac{4-n}{2}}\int_{\Omega}U_{x,\lambda}^{\frac{n+4}{n-4}}H(x,y)\,\mathrm{d}y
			=\mathfrak {a}_nc_n^{\frac{n+4}{n-4}}\lambda^{4-n}R(x)+\lambda^{4-n}O\big(\rho d(x)^{3-n}\big)+\lambda^{4-n}O\big(d(x)^{4-n}(\lambda\rho)^{-4}\big).
	\end{equation*}
	 Choosing $\rho=d(x)^\frac{1}{5}\lambda^{-\frac{4}{5}}$ when $\lambda$ is large enough, we obtain the second term estimate in \eqref{A1.15},
	\begin{equation}\label{A1.24}
		\lambda^{\frac{4-n}{2}}\int_{\Omega}U_{x,\lambda}^{\frac{n+4}{n-4}}H(x,y)\,\mathrm{d}y=\mathfrak {a}_nc_n^{\frac{n+4}{n-4}}	\lambda^{4-n}R(x)+O\big((d(x)\lambda)^{\frac{16}{5}-n}\big).
	\end{equation}
	We turn to the last term in \eqref{A1.15}. Using \eqref{A1.3} and \eqref{A1.14}, we have
	\begin{equation}\label{A1.25}
		\begin{aligned}
			\int_{\Omega}U_{x,\lambda}^{\frac{n+4}{n-4}}f_{x,\lambda}\,\mathrm{d}y\leq|f_{x,\lambda}|_{L^{\infty}(\Omega)}\int_{\mathbb{R}^n}U_{x,\lambda}^{\frac{n+4}{n-4}}\,\mathrm{d}y=O\big((d(x)\lambda)^{2-n}\big).
		\end{aligned}
	\end{equation}
	Consequently, substituting \eqref{A1.16}, \eqref{A1.24}, \eqref{A1.25} into \eqref{A1.15}, we conclude that \eqref{A1.10} holds.
	
	We now prove \eqref{A1.10+}. From \eqref{A1.13}, we know that
	\begin{equation}\label{A1.26}
			\int_{\Omega}V|PU_{x,\lambda}|^{2}\,\mathrm{d}y
			=\int_{\Omega}VU_{x,\lambda}^{2}\,\mathrm{d}y+\bigg(-2\int_{\Omega}VU_{x,\lambda}\theta_{x,\lambda}\,\mathrm{d}y+\int_{\Omega}V\theta_{x,\lambda}^2\,\mathrm{d}y\bigg).
	\end{equation}
	Recalling that 
	$0\leq \theta_{x,\lambda}(y)\leq U_{x,\lambda}(y)$ for all $y\in\Omega$ (Proposition 2.1 in \cite{AH2006}), which yields
	\begin{equation}\label{A1.27}
			\bigg|-2\int_{\Omega}VU_{x,\lambda}\theta_{x,\lambda}\,\mathrm{d}y+\int_{\Omega}V\theta_{x,\lambda}^2\,\mathrm{d}y\bigg|
			\leq 2|V|_{L^\infty(\Omega)}|\theta_{x,\lambda}|_{L^\infty(\Omega)}\int_{\Omega}U_{x,\lambda}\,\mathrm{d}y
			=O\big((\lambda d(x))^{4-n}\big).
	\end{equation}
	For $n=8$, let $B_\tau(x)\subset \Omega \subset B_R(x)$, where $\tau$ is chosen sufficiently small to satisfy two conditions: $o_{\tau}(1)\to0$ as $\tau\to0$, and $\tau\lambda\to\infty$. We further extend $V$ by zero to the complement $B_R(x)\setminus\Omega$. Then we have
	\begin{equation*}
		\begin{aligned}
			\int_{\Omega\setminus B_{\tau}(x)}VU_{x,\lambda}^{2}\,\mathrm{d}y=&\int_{B_R(x)\setminus B_{\tau}(x)}VU_{x,\lambda}^{2}\,\mathrm{d}y
			\leq|V|_{L^{\infty}(B_R(x)\setminus B_{\tau}(x))}\int_{B_R(x)\setminus B_{\tau}(x)}U_{x,\lambda}^{2}\,\mathrm{d}y\\
			=&\omega_8c_8^2|V|_{L^{\infty}(B_R(x)\setminus B_{\tau}(x))}\lambda^{-4}\int_{\tau\lambda}^{R\lambda}\frac{t^{7}}{(1+t^2)^{4}}\,\mathrm{d}t
			=O\big(\lambda^{-4}\log(R/\tau)\big)
		\end{aligned}
	\end{equation*}
	and 
	\begin{equation}\label{A1.29}
		\begin{aligned}
			\int_{B_{\tau}(x)}VU_{x,\lambda}^{2}\,\mathrm{d}y=&\omega_8c_8^2(V(x)+o_{\tau}(1))\int_{0}^\tau \frac{\lambda^4r^{7}}{(1+\lambda^2r^2)^{4}}\,\mathrm{d}r\\
			=&\omega_8c_8^2V(x)\lambda^{-4}\log\lambda+O\bigg(\frac{\log\tau}{\lambda^4}\bigg)+o_{\tau}(1)O\bigg(\frac{\log\lambda}{\lambda^{4}}\bigg).
		\end{aligned}
	\end{equation}
	By choosing $\tau=\frac{1}{\log\lambda}$ and taking into account \eqref{A1.26}–\eqref{A1.29}, the first case in \eqref{A1.10+} is proved.
	
	For $n\geq 9$, we select a sequence $ \eta= \eta_\lambda\leq d(x)$ such that $ \eta \to 0$ and $\eta\lambda \to +\infty$ as $\lambda \to +\infty$. By the continuity of $V$,
	\begin{equation}\label{A1.30}
		\begin{aligned}
			\int_{\Omega}VU_{x,\lambda}^2\,\mathrm{d}y=&(V(x)+o(1))\int_{B_{\eta}(x)}U_{x,\lambda}^2\,\mathrm{d}y+\int_{\Omega\setminus B_{\eta}(x)}VU_{x,\lambda}^2\,\mathrm{d}y\\
			=&\lambda^{-4}\mathfrak {b}_nc_n^2V(x)+o(\lambda^{-4})+O\bigg(\int_{\Omega\setminus B_{\eta}(x)}U_{x,\lambda}^2\,\mathrm{d}y\bigg)\\
			%=&\lambda^{-4}\mathfrak {b}_nc_n^2V(x)+o(\lambda^{-4})+O(\lambda^{-4}(\eta\lambda)^{8-n})\\
			=&\lambda^{-4}\mathfrak {b}_nc_n^2V(x)+o(\lambda^{-4}).\\
		\end{aligned}
	\end{equation}
	Combining \eqref{A1.26}, \eqref{A1.27} and \eqref{A1.30}, the proof of \eqref{A1.10+} is finished.
	
	We now prove \eqref{A1.10++}. Thanks to $2^\star>2$, Taylor's expansion tells us that
	$$|b^{2^\star}-(b-a)^{2^\star}-2^\star b^{2^\star-1}a|\leq\frac{2^\star(2^\star-1)}{2}b^{2^\star-2}a^2\hspace{4mm}\text{for any}\hspace{2mm}a\in[0,b].$$
	By choosing $b=U_{x,\lambda}$ and $a=\theta_{x,\lambda}$, it follows that
	\begin{equation}\label{A1.31}
		\big|PU_{x,\lambda}^{2^\star}-U_{x,\lambda}^{2^\star}+2^\star U_{x,\lambda}^{2^\star-1}\theta_{x,\lambda}\big|\leq\frac{2^\star(2^\star-1)}{2}U_{x,\lambda}^{2^\star-2}\theta_{x,\lambda}^2.
	\end{equation}
	Then we shall estimate the right-hand side of \eqref{A1.31}. Let $\delta=\frac{d(x)}{2}$. We first treat the case $n=8$ on $B_{\delta}(x)$, for which we have
	\begin{equation}\label{A2.1}
			\int_{B_{\delta}(x)}U_{x,\lambda}^{2^\star-2}\theta_{x,\lambda}^2\,\mathrm{d}y
			\leq O\big((d(x)\lambda)^{-8}\big)\int_{0}^{\delta\lambda}\frac{t^{7}}{(1+t^2)^4}\,\mathrm{d}t
			=O\bigg(
			\frac{\log\lambda}{(d(x)\lambda)^{8}}\bigg)+o\bigg(	\frac{\log\lambda}{(d(x)\lambda)^{8}}\bigg).
	\end{equation}
Turning to the case $n\geq9$ on $B_{\delta}(x)$, it holds that 
	\begin{equation*}
			\int_{B_{\delta}(x)}U_{x,\lambda}^{2^\star-2}\theta_{x,\lambda}^2\,\mathrm{d}y\leq
			|\theta_{x,\lambda}|^2_{L^\infty(B_{\delta})}\int_{B_{\delta}(x)}U_{x,\lambda}^{2^\star-2}\,\mathrm{d}y
			=O\big((d(x)\lambda)^{-n}\big).
	\end{equation*}
	Next, we estimate $U_{x,\lambda}^{2^\star-2}\theta_{x,\lambda}^2$ on
	$\Omega\setminus B_{\delta}(x)$. Using the H\"older inequality and the fact that
	$|\theta_{x,\lambda}|_{L^{2^\star}(\Omega)}=O\big((\lambda d(x))^{-\frac{n-4}{2}}\big)$ (Proposition 2.1 in \cite{AH2006}), we derive
	\begin{equation}\label{A2.2}
		\int_{\Omega\setminus B_\delta(x)}U_{x,\lambda}^{2^\star-2}\theta_{x,\lambda}^2\,\mathrm{d}y\leq\bigg(\int_{\Omega}\theta_{x,\lambda}^{2^\star}\,\mathrm{d}y\bigg)^{\frac{2}{2^\star}}\bigg(\int_{\mathbb{R}^n\setminus B_{\delta}(x)}U_{x,\lambda}^{2^\star},\mathrm{d}y\bigg)^{\frac{2^\star-2}{2^\star}}=O\big((d(x)\lambda)^{-n}\big).
	\end{equation}
	Hence, by \eqref{A2.1}–\eqref{A2.2} we arrive at
	\begin{equation}\label{A1.35}
		\int_{\Omega}U_{x,\lambda}^{2^\star-2}\theta_{x,\lambda}^2\,\mathrm{d}y=
		\left\lbrace
		\begin{aligned}
			O\bigg(
			\frac{\log\lambda}{(d(x)\lambda)^{8}}\bigg)+o\bigg(	\frac{\log\lambda}{(d(x)\lambda)^{8}}\bigg)&\hspace{9mm}\text{if}\hspace{2mm}n=8,\\
			O\big((d(x)\lambda)^{-n}\big)&\hspace{9mm}\text{if}\hspace{2mm}n\geq9,
		\end{aligned}
		\right.
	\end{equation}
	In light of \eqref{A1.16}, \eqref{A1.31} and \eqref{A1.35}, the estimate \eqref{A1.10++} follows.
\end{proof}

\begin{cor}\label{cor2.2}
	As $\varepsilon\to0^+$, we have
	\begin{equation}\label{A1.39}
		S(\varepsilon V)\leq\left\lbrace
		\begin{aligned}
			\mathcal{S}_2-e^{-\frac{c_8^2 }{10\Phi_8\varepsilon}(1+o(1))}\hspace{9mm}\text{if}\hspace{2mm}n=8,\\
			\mathcal{S}_2-\mathfrak{C}_n\Phi_n\varepsilon^{\frac{n-4}{n-8}}+o(\varepsilon^{\frac{n-4}{n-8}})\hspace{9mm}\text{if}\hspace{2mm}n\geq 9,
		\end{aligned}
		\right.
	\end{equation}
\end{cor}	
\begin{proof}
	Observe that, following the argument of \cite[(2.8)]{Rey1990} and using the Navier boundary condition, an iterated application of the maximum principle implies
	\begin{equation}\label{A3.1+}
		d(x)^{4-n}\lesssim R(x)\lesssim d(x)^{4-n}.
	\end{equation}
	(Note that this bound uses the $C^2$ assumption on $\partial\Omega$.) For $n=8$, due to the continuity of $V(x)$ on the bounded domain 
	$\Omega$ and \eqref{A3.1+}, we can further extend the function $R(x)^{-1}|V(x)|$ on $\overline{{\mathcal{N}(V)}}$, and it vanishes on $\partial\mathcal{N}(V)$. Consequently, there exists a point $z_0\in\mathcal{N}(V)$ such that $\Phi_8=R(z_0)^{-1}|V(z_0)|$. By substituting  $x=z_0$ and $\lambda=e^{\frac{1}{4}}e^{\frac{\mathfrak {a}_8c_8^2\big(R(z_0)+o(1)\big)}{\mathfrak {b}_8\big(|V(z_0)|+o(1)\big)\varepsilon}}$ into \eqref{A1.10+++}, where $\lambda$ is the extremum of $\frac{1}{\mathcal{S}_2\lambda^{4}}(\mathfrak {a}_8c_8^{4}(R(z_0)+o(1))+\varepsilon \mathfrak {b}_8c_8^2(V(z_0)+o(1))\log\lambda)$, we have
	$$S(\varepsilon V)\leq S_{\varepsilon V}[PU_{x,\lambda}]=\mathcal{S}_2-\frac{\varepsilon \mathfrak {b}_8c_8^2|V(z_0)|(1+o(1))}{4 e\mathcal{S}_2}e^{-\frac{c_8^2\big(R(z_0)+o(1)\big)}{10\varepsilon\big(|V(z_0)|+o(1)\big)}}.$$
	In conjunction with the fact that
	\begin{equation}\label{A1.41}
		\varepsilon b e^{-\frac{a}{\varepsilon}}=e^{-\frac{a}{\varepsilon}+o(\frac{1}{\varepsilon})}
	\end{equation}
	for any $a\geq0$ and $b>0$, we derive 
	$$S(\varepsilon V)\leq \mathcal{S}_2-e^{-\frac{c_8^2R(z_0)}{10|V(z_0)|\varepsilon}\big(1+o(1)\big)},$$
	which implies \eqref{A1.39}.
	
	On the other hand, for $n\geq9$, we can show $R(x)^{-\frac{4}{n-8}}|V(x)|^{\frac{n-4}{n-8}}$ is bounded on $\Omega$. Since $V=0$ on $\partial\mathcal{N}(V)\setminus\partial\Omega$, the function $R(x)^{-\frac{4}{n-8}}|V(x)|^{\frac{n-4}{n-8}}$ can be extended to a continuous function on $\overline{{\mathcal{N}(V)}}$ vanishing on $\partial\mathcal{N}(V)$. Consequently, there exists a point 
	$z_0\in\mathcal{N}(V)$ such that $\Phi_n=R(z_0)^{-\frac{4}{n-8}}|V(z_0)|^{\frac{n-4}{n-8}}$. We substitute $x=z_0$ and $\lambda=\big(\frac{(n-4)\mathfrak {a}_nc_n^{\frac{8}{n-4}}R(z_0)}{4\varepsilon \mathfrak {b}_n|V(z_0)|}\big)^{\frac{1}{n-8}}$ into \eqref{A1.10+++}, where 
	$\lambda$ is the extremum of the expression $\frac{\mathfrak {a}_nc_n^{2^\star}R(z_0)}{{\lambda^{n-4}}}+\frac{\varepsilon \mathfrak {b}_nc_n^2V(z_0)}{\lambda^{4}}$. The proof of the corollary is now complete.
\end{proof}

\section{Lower Bound Estimate}
In this section, we will present the lower bound estimate for the minimal energy $S(\varepsilon V)$. We begin by an asymptotic decomposition for a general almost minimizing sequence of $S(\varepsilon V)$.
\begin{Prop}\label{Prop1}
	Let $\{u_\varepsilon\}\subset \mathcal{H}$ be a sequence of functions satisfying
	\begin{equation*}
		S_{\varepsilon V}[u_\varepsilon]\to \mathcal{S}_2\hspace{4mm}\text{and}\hspace{4mm}
		\int_{\Omega}|u_\varepsilon|^{2^\star}dx=\mathcal{S}_2^{\frac{n}{4}}.
	\end{equation*}
	Then, up to a subsequence,
	\begin{equation}\label{toukui}
		u_\varepsilon=\alpha_\varepsilon(PU_{x_\varepsilon,\lambda_\varepsilon}+v_\varepsilon)
	\end{equation}
	with
	\begin{equation*}
		\alpha_\varepsilon\to s\hspace{6mm}\text{for some}\hspace{2mm}s\in\{-1,1\},
	\end{equation*}
	\begin{equation}\label{dingyi}
		x_\varepsilon\to x_0\hspace{6mm}\text{for some}\hspace{2mm}x_0\in\overline{\Omega},
	\end{equation}
	\begin{equation}\label{nanjing}
		\lambda_\varepsilon d_\varepsilon\to+\infty,
	\end{equation}
	\begin{equation*}
		\|v_\varepsilon\|\to0,
	\end{equation*}
	where $v_\varepsilon\in T_{x,\lambda}^\perp$ and $d_\varepsilon=\dist(x_\varepsilon,\partial\Omega)$.
\end{Prop}
\begin{proof}
	The assumptions imply that sequence $\{u_\varepsilon\}$ is bounded in $\mathcal{H}$, hence, up to a subsequence, we may assume $u_\varepsilon\rightharpoonup u_0$ for some $u_0\in\mathcal{H}$. By the argument given in \cite[Proof of Proposition 3.1, Step 1]{FKK2021}, together with the fact that $\mathcal{S}_2$ is never achieved except when $\Omega=\mathbb{R}^n$, we conclude that $u_0\equiv0$. 
	
	By the Rellich theorem, $u_\varepsilon\rightarrow0$ strongly in $L^2(\Omega)$, in particular, $\varepsilon\int_{\Omega}V|u_\varepsilon|^2dx=o(1)$. Therefore,
	\begin{equation*}
		\frac{\int_{\Omega}|\Delta u_\varepsilon|^2dx}{\big(\int_{\Omega}|u_\varepsilon|^{2^\star}dx\big)^{2/{2^\star}}}\rightarrow\mathcal{S}_2.
	\end{equation*}
	Thus, the $u_\varepsilon$, extended by zero to functions in $D^{2,2}(\mathbb{R}^n)$, form a minimizing sequence for the Sobolev quotient. By a theorem of Lions \cite[Theorem I.1]{Lions1985}, there exist $\{z_\varepsilon\}\subset\mathbb{R}^n$ and $\{\mu_\varepsilon\}\subset\mathbb{R}^+$ such that, along a subsequence,
	\begin{equation*}
		\mu_\varepsilon^{-\frac{n-4}{2}}u_\varepsilon(\mu_\varepsilon^{-1}\cdot+z_\varepsilon)\rightarrow\beta U_{0,1}
	\end{equation*}
	in  $D^{2,2}(\mathbb{R}^n)$ for some $\beta\in\mathbb{R}$. By the normalization condition $\int_{\Omega}|u_\varepsilon|^{2^\star}dx=\mathcal{S}_2^{\frac{n}{4}}$, we have $\beta\in\{\pm1\}$. Now, a change of variables $y=z_\varepsilon+\mu_\varepsilon^{-1}x$ (which preserves the $D^{2,2}(\mathbb{R}^n)$-norm) yields
	\begin{equation}\label{kaihui}
		u_\varepsilon(y)=\beta U_{z_\varepsilon,\mu_\varepsilon}(y)+\sigma_\varepsilon,
	\end{equation}
	where $\sigma_\varepsilon\rightarrow0$ in $D^{2,2}(\mathbb{R}^n)$. 
	
	Note that, since the boundary of $\Omega$ is $C^2$, the fact
	\begin{equation*}
		\int_{\mathbb{R}^n}U_{0,1}^{2^\star}=\int_{\Omega}|u_\varepsilon|^{2^\star}dx=\int_{\Omega}U_{z_\varepsilon,\mu_\varepsilon}^{2^\star}dx+o(1)
	\end{equation*}
	implies $\mu_\varepsilon\dist(z_\varepsilon,\mathbb{R}^n\setminus\Omega)\rightarrow+\infty$. In particular, after passing to a subsequence, $z_\varepsilon\rightarrow x_0\in\overline{\Omega}$. We now make the necessary modifications to derive \eqref{toukui} from \eqref{kaihui}. A simple calculation yields
	\begin{equation*}
			\|U_{z_\varepsilon,\mu_\varepsilon}-PU_{z_\varepsilon,\mu_\varepsilon}\|^2
			=\|U_{z_\varepsilon,\mu_\varepsilon}\|^2+\| PU_{z_\varepsilon,\mu_\varepsilon}\|^2-2\int_{\Omega}(\Delta^2U_{z_\varepsilon,\mu_\varepsilon})PU_{z_\varepsilon,\mu_\varepsilon}dy.
	\end{equation*}
	From \eqref{A1.13}, we know
	\begin{equation*}
		\|PU_{z_\varepsilon,\mu_\varepsilon}\|^2=\| U_{z_\varepsilon,\mu_\varepsilon}\|^2-\int_{\Omega}\theta_{z_\varepsilon,\mu_\varepsilon}U_{z_\varepsilon,\mu_\varepsilon}^{2^\star-1}dy=\| U_{z_\varepsilon,\mu_\varepsilon}\|^2+O\big((\mu_\varepsilon\tilde{d}_\varepsilon)^{4-n}\big)
	\end{equation*}
	and
	\begin{equation*}
		-2\int_{\Omega}(\Delta^2U_{z_\varepsilon,\mu_\varepsilon})PU_{z_\varepsilon,\mu_\varepsilon}dy=-2\int_{\Omega}U_{z_\varepsilon,\mu_\varepsilon}^{2^\star}dy+2\int_{\Omega}\theta_{z_\varepsilon,\mu_\varepsilon}U_{z_\varepsilon,\mu_\varepsilon}^{2^\star-1}=-2\| U_{z_\varepsilon,\mu_\varepsilon}\|^2+O\big((\mu_\varepsilon\tilde{d}_\varepsilon)^{4-n}\big),
	\end{equation*}
	where $\tilde{d}_\varepsilon=\dist(z_\varepsilon,\mathbb{R}^n\setminus\Omega)$. Consequently,
	\begin{equation}\label{fuluu}
		\|U_{z_\varepsilon,\mu_\varepsilon}-PU_{z_\varepsilon,\mu_\varepsilon}\|\rightarrow0\hspace{6mm}\text{as}\hspace{2mm}\varepsilon\rightarrow0^+.
	\end{equation}
	
	The key ingredient is provided by \cite[Lemma 2.2]{EA2003}, which extends the corresponding result of Bahri and Coron \cite[Proposition 7]{BC1988}. Specifically, suppose that $u\in\mathcal{H}$ with $\|u\|=\mathcal{S}_2^{\frac{n}{4}}$ satisfies
	\begin{equation}\label{tiaojian}
		\inf\big\{\|u-PU_{x,\lambda}\|: x\in\Omega,  \lambda d(x)>\tilde{\eta}^{-1}\big\}<\tilde{\eta}
	\end{equation}
	for some $\tilde{\eta}>0$. Then, if $\tilde{\eta}$ is small enough, the minimization problem
	\begin{equation}\label{chengzi}
		\inf_{x,\lambda,\alpha}\big\{\|u-\alpha PU_{x,\lambda}\|: x\in\Omega, \lambda d(x)>(4\tilde{\eta})^{-1},\alpha\in(1/2,2)\big\}
	\end{equation}
	has a unique solution.
	
	From \eqref{kaihui} and \eqref{fuluu},
	\begin{equation*}
	    \|u_\varepsilon-PU_{z_\varepsilon,\mu_\varepsilon}\|\leq\|U_{z_\varepsilon,\mu_\varepsilon}-PU_{z_\varepsilon,\mu_\varepsilon}\|+\|\sigma_\varepsilon\|\rightarrow0\hspace{6mm}\text{as}\hspace{2mm}\varepsilon\rightarrow0^+,
	\end{equation*}
	so that \eqref{tiaojian} is satisfied by $u_\varepsilon$ for all $\varepsilon$ small enough, with a constant $\tilde{\eta}_\varepsilon$ tending to zero. We thus obtain the desired decomposition
	\begin{equation*}
		u_\varepsilon=\alpha_\varepsilon(PU_{x_\varepsilon,\lambda_\varepsilon}+v_\varepsilon)
	\end{equation*}
	by taking $(x_\varepsilon,\lambda_\varepsilon,\alpha_\varepsilon)$ to be the solution to \eqref{chengzi} and $\alpha_\varepsilon^{-1}u_\varepsilon-PU_{x_\varepsilon,\lambda_\varepsilon}=v_\varepsilon\in T_{x,\lambda}^\perp$. To verify the claimed asymptotic behavior of the parameters, note that since $\tilde{\eta}_\varepsilon\rightarrow0$, by the definition of minimization problem \eqref{chengzi}, we have $\|v_\varepsilon\|<\tilde{\eta}_\varepsilon\rightarrow0$ and $\lambda_\varepsilon d_\varepsilon>(4\tilde{\eta}_\varepsilon)^{-1}\rightarrow+\infty$. Since $\Omega$ is bounded, the convergence $x_\varepsilon\rightarrow x_0\in\overline{\Omega}$ is ensured by passing to a suitable subsequence. 
	
	Finally, using the normalization condition, we get
	\begin{equation*}
		\int_{\mathbb{R}^n}U_{0,1}^{2^\star}dy=\int_{\Omega}|u_\varepsilon|^{2^\star}dy=|\alpha_\varepsilon|^{2^\star}\int_{\Omega}PU_{x_\varepsilon,\lambda_\varepsilon}^{2^\star}dy+o(1)=|\alpha_\varepsilon|^{2^\star}\int_{\mathbb{R}^n}U_{0,1}^{2^\star}+o(1),
	\end{equation*}
	which implies $\alpha_\varepsilon=\pm1+o(1)$. The proof is finished.
\end{proof}

From now on we will assume that $\{u_\varepsilon\}$ satisfies \eqref{A1.1}. In particular, the assumptions of Proposition \ref{Prop1} are satisfied. We now expand $S_{\varepsilon V}[u_\varepsilon]$ by using the decomposition \eqref{toukui} of $u_\varepsilon$. For simplicity, in what follows, we omit the subscript $\varepsilon$ from $\alpha_\varepsilon$, $x_\varepsilon$, $\lambda_\varepsilon$, $d_\varepsilon$ and $v_\varepsilon$.
\begin{Prop}\label{homogeneity}
	Let $\{u_\varepsilon\}\subset\mathcal{H}$ be a sequence of functions satisfying \eqref{toukui} and \eqref{nanjing}. Then
	\begin{equation}\label{A3.1}
		|\alpha|^{-2}\int_{\Omega}|\Delta u_\varepsilon|^2\,\mathrm{d}y=\int_{\Omega}|\Delta PU_{x,\lambda}|^2\,\mathrm{d}y+\int_{\Omega}|\Delta v|^2\,\mathrm{d}y,
	\end{equation}
	\begin{equation}\label{A3.2}
		|\alpha|^{-2^\star}\int_{\Omega}|u_\varepsilon|^{2^\star}\mathrm{d}y=\int_{\Omega}PU_{x,\lambda}^{2^\star}\,\mathrm{d}y+\frac{2^\star(2^\star-1)}{2}\int_{\Omega}U_{x,\lambda}^{2^\star-2}v^2\,\mathrm{d}y+o\bigg(\int_{\Omega}|\Delta v|^2\,\mathrm{d}y+(\lambda d)^{4-n}\bigg),
	\end{equation}
	\begin{equation}\label{A3.3}
		|\alpha|^{-2}\varepsilon\int_{\Omega}Vu_\varepsilon^2\,\mathrm{d}y=\varepsilon\int_{\Omega}VPU_{x,\lambda}^2\,\mathrm{d}y+O\bigg(\varepsilon\int_{\Omega}|\Delta v|^2\,\mathrm{d}y+\varepsilon\sqrt{\int_{\Omega}|\Delta v|^2\,\mathrm{d}y}\sqrt{\int_{\Omega}|V|PU_{x,\lambda}^2\,\mathrm{d}y}\bigg).
	\end{equation}
	In particular, 
	\begin{equation}\begin{aligned}\label{A3.4}
			S_{\varepsilon V}[u_\varepsilon]=S_{\varepsilon V}&[PU_{x,\lambda}]+I[v]+O\bigg(\varepsilon\sqrt{\int_{\Omega}|\Delta v|^2\,\mathrm{d}y}\sqrt{\int_{\Omega}|V|PU_{x,\lambda}^2\,\mathrm{d}y}\bigg)\\
			&+o\bigg(\int_{\Omega}|\Delta v|^2\,\mathrm{d}y+(\lambda d
			)^{4-n}\bigg),
		\end{aligned}
	\end{equation}
	where
	$$I[v]:=\bigg(\int_{\Omega}U_{x,\lambda}^{2^\star}\,\mathrm{d}y\bigg)^{-\frac{2}{2^\star}}\bigg(\int_{\Omega}|\Delta v|^2\,\mathrm{d}y-(2^\star-1)\int_{\Omega}U_{x,\lambda}^{2^\star-2}v^2\,\mathrm{d}y\bigg).$$
\end{Prop}
\begin{proof}
	First, using the orthogonality condition $v\in T_{x,\lambda}^\perp$, one can derive \eqref{A3.1} directly.
	
	We now prove \eqref{A3.2}. Since $\alpha^{-1}u_\varepsilon=U_{x,\lambda}+(v-\theta_{x,\lambda})$, we can apply the associated pointwise estimates to get
	$$\begin{aligned}
		&\bigg||\alpha|^{-2^\star}|u_\varepsilon|^{2^\star}-U_{x,\lambda}^{2^\star}-2^\star U_{x,\lambda}^{2^\star-1}(v-\theta_{x,\lambda})-\frac{2^\star(2^\star-1)}{2}U_{x,\lambda}^{2^\star-2}(v-\theta_{x,\lambda})^2\bigg|\\
		\leq&C_1\bigg(|v-\theta_{x,\lambda}|^{2^\star}+|v-\theta_{x,\lambda}|^{2^\star-(2^\star-3)_+}U_{x,\lambda}^{(2^\star-3)_+}\bigg),
	\end{aligned}$$
	where $(2^\star-3)_+=\max\{2^\star-3,0\}$. Using \eqref{A1.31}, we obtain  
	\begin{equation*}\begin{aligned}
			&\bigg||\alpha|^{-2^\star}|u_\varepsilon|^{2^\star}-PU_{x,\lambda}^{2^\star}-2^\star U_{x,\lambda}^{2^\star-1}v-\frac{2^\star(2^\star-1)}{2}U_{x,\lambda}^{2^\star-2}v^2\bigg|\\
			\leq&C_2\big(|v-\theta_{x,\lambda}|^{2^\star}+|v-\theta_{x,\lambda}|^{2^\star-(2^\star-3)_+}U_{x,\lambda}^{(2^\star-3)_+}+U_{x,\lambda}^{2^\star-2}\theta_{x,\lambda}|v|+U_{x,\lambda}^{2^\star-2}\theta_{x,\lambda}^2\big)\\
			%\leq&c_3\big(|v|^{2^\star}+|\theta_{x,\lambda}|^{2^\star}+|v|^{2^\star-(2^\star-3)_+}U_{x,\lambda}^{(2^\star-3)_+}+|\theta_{x,\lambda}|^{2^\star-(2^\star-3)_+}U_{x,\lambda}^{(2^\star-3)_+}+U_{x,\lambda}^{2^\star-2}\theta_{x,\lambda}|v|+U_{x,\lambda}^{2^\star-2}\theta_{x,\lambda}^2\big)\\
			\leq&C_3\big(|v|^{2^\star}+|v|^{2^\star-(2^\star-3)_+}U_{x,\lambda}^{(2^\star-3)_+}+U_{x,\lambda}^{2^\star-2}\theta_{x,\lambda}|v|+U_{x,\lambda}^{2^\star-2}\theta_{x,\lambda}^2\big).\\
		\end{aligned}
	\end{equation*}
	Since $v\in T_{x,\lambda}^\perp$, it holds that
	\begin{equation*}
		\int_{\Omega}U_{x,\lambda}^{2^\star-1}v\,\mathrm{d}y=\int_{\Omega}\Delta U_{x,\lambda}\Delta v\,\mathrm{d}y=\int_{\Omega}\Delta PU_{x,\lambda}\Delta v\,\mathrm{d}y=0.
	\end{equation*}
	Then, by the H\"older inequality,
	\begin{equation*}
		\begin{aligned}
			&\bigg|\int_{\Omega}\bigg(|\alpha|^{-2^\star}|u_\varepsilon|^{2^\star}-PU_{x,\lambda}^{2^\star}-\frac{2^\star(2^\star-1)}{2}U_{x,\lambda}^{2^\star-2}v^2\bigg)\,\mathrm{d}y\bigg|\\
			\leq&C_3\bigg[\int_{\Omega}|v|^{2^\star}\,\mathrm{d}y+\bigg(\int_{\Omega}|v|^{2^\star}\,\mathrm{d}y\bigg)^{\frac{2^\star-(2^\star-3)_+}{2^\star}}\bigg(\int_{\Omega}|U_{x,\lambda}|^{2^\star}\,\mathrm{d}y\bigg)^{\frac{(2^\star-3)_+}{2^\star}}\\
			&+\bigg(\int_{\Omega}|U_{x,\lambda}|^{\frac{2^\star(2^\star-2)}{2^\star-1}}\theta_{x,\lambda}^{\frac{2^\star}{2^\star-1}}\,\mathrm{d}y\bigg)^{\frac{2^\star-1}{2^\star}}\bigg(\int_{\Omega}|v|^{2^\star}\,\mathrm{d}y\bigg)^{\frac{1}{2^\star}}+\int_{\Omega}U_{x,\lambda}^{2^\star-2}\theta_{x,\lambda}^2\,\mathrm{d}y\bigg]\\
			\leq&C_4\bigg[\bigg(\int_{\Omega}|\Delta v|^2\,\mathrm{d}y\bigg)^{\frac{2^\star-(2^\star-3)_+}{2}}+\bigg(\int_{\Omega}|U_{x,\lambda}|^{\frac{2^\star(2^\star-2)}{2^\star-1}}\theta_{x,\lambda}^{\frac{2^\star}{2^\star-1}}\,\mathrm{d}y\bigg)^{\frac{2^\star-1}{2^\star}}\bigg(\int_{\Omega}|\Delta v|^2\,\mathrm{d}y\bigg)^{\frac12}+\int_{\Omega}U_{x,\lambda}^{2^\star-2}\theta_{x,\lambda}^2\,\mathrm{d}y\bigg],
		\end{aligned}
	\end{equation*}
	where we have used the Sobolev inequality and the fact $\|v\|\to0$ as $\varepsilon\to 0^+$ in the last inequality. By Lemma \ref{ALem1} and \eqref{A1.35}, we know 
	$$\bigg(\int_{\Omega}|U_{x,\lambda}|^{\frac{2^\star(2^\star-2)}{2^\star-1}}\theta_{x,\lambda}^{\frac{2^\star}{2^\star-1}}\,\mathrm{d}y\bigg)^{\frac{2^\star-1}{2^\star}}=o\big(\lambda d\big)^{\frac{4-n}{2}},$$
	$$\int_{\Omega}U_{x,\lambda}^{2^\star-2}\theta_{x,\lambda}^2\,\mathrm{d}y=o\big((\lambda d)^{4-n}\big).$$
	Therefore, we can achieve that as $\varepsilon\to 0^+$,
	$$	\bigg|\int_{\Omega}\bigg(|\alpha|^{-2^\star}|u_\varepsilon|^{2^\star}-PU_{x,\lambda}^{2^\star}-\frac{2^\star(2^\star-1)}{2}U_{x,\lambda}^{2^\star-2}v^2\bigg)\,\mathrm{d}y\bigg|=o\bigg(\int_{\Omega}|\Delta v|^2\,\mathrm{d}y+(\lambda d
	)^{4-n}\bigg),$$
	which means that \eqref{A3.2} holds true.
	
	We now prove \eqref{A3.3}. We write
	\begin{equation}\label{A3.5}
		|\alpha|^{-2}\int_{\Omega}Vu_\varepsilon^2\,\mathrm{d}y=\int_{\Omega}VPU_{x,\lambda}^2\,\mathrm{d}y+2\int_{\Omega}VPU_{x,\lambda}v\,\mathrm{d}y+\int_{\Omega}Vv^2\,\mathrm{d}y.\end{equation}
	Applying the H\"older inequality and the Sobolev inequality, we have
	$$
		\bigg|\int_{\Omega}VPU_{x,\lambda}v\,\mathrm{d}y\bigg|\leq \mathcal{S}_2^{-\frac{1}{2}}\bigg(\int_{\Omega}|V|PU_{x,\lambda}^2\,\mathrm{d}y\bigg)^{\frac{1}{2}}|V|_{L^{\frac{n}{4}}(\Omega)}^{\frac{1}{2}}\|v\|$$
	and 
	$$\bigg|\int_{\Omega}Vv^2\,\mathrm{d}y\bigg|\leq|V|_{L^{\frac{n}{4}}(\Omega)}|v|_{L^{2^\star}(\Omega)}^2\leq \mathcal{S}_2^{-1}|V|_{L^{\frac{n}{4}}(\Omega)}\|v\|^2. $$
	Inserting the above estimates into \eqref {A3.5} yields \eqref {A3.3}.
	
	Finally, by the homogeneity of the quotient functional $S_{\varepsilon V}[u]$, we can see that $S_{\varepsilon V}[u_\varepsilon]=S_{\varepsilon V}[|\alpha|^{-1}u_\varepsilon]$. Consequently, a straightforward Taylor expansion of $S_{\varepsilon V}[u_\varepsilon]$, together with \eqref{A3.1}–\eqref{A3.3}, yields \eqref{A3.4}.
\end{proof}

We shall use the following coercivity inequality from \cite[Proposition 2]{Konig2023} to control the remainder terms and refine the expansion of $u_\varepsilon$.
\begin{Prop}\label{Prop3.3}
	Let $v\in T_{x,\lambda}^\perp$. Then
	\begin{equation*}
		\int_{\Omega}|\Delta v|^2\,\mathrm{d}y-(2^\star-1)\int_{\Omega}U_{x,\lambda}^{2^\star-2}v^2\,\mathrm{d}y\geq\frac{8}{n+6}\int_{\Omega}|\Delta v|^2\,\mathrm{d}y.
	\end{equation*}
\end{Prop}

\begin{cor}\label{cor3.4}
	For any $\varepsilon>0$ small enough, if $n=8$, we have
	\begin{equation}\label{kele}
		\begin{aligned}
			0\geq(1+o(1))(\mathcal{S}_2-S(\varepsilon V))+&\frac{1}{\mathcal{S}_2\lambda^{4}}\bigg(\mathfrak {a}_8c_8^{4}R(x)+\varepsilon \mathfrak {b}_8c_8^2V(x)\log\lambda\bigg)\\
			&+c\int_{\Omega}|\Delta v|^2\,\mathrm{d}y+o(\varepsilon\lambda^{-4}\log\lambda)+o\big((\lambda d)^{-4}\big),
		\end{aligned}
	\end{equation}
	and, if $n\geq 9$, we have
	\begin{equation}\label{baoshifu}
		\begin{aligned}
			0\geq(1+o(1))(\mathcal{S}_2-S(\varepsilon V))+&\mathcal{S}_2^{1-\frac{n}{4}}\bigg(\frac{\mathfrak {a}_nc_n^{2^\star}R(x)}{{\lambda^{n-4}}}+\varepsilon \mathfrak {b}_nc_n^2V(x)\lambda^{-4}\bigg)\\
			&+c\int_{\Omega}|\Delta v|^2\,\mathrm{d}y+o(\varepsilon\lambda^{-4})+o\big((\lambda d)^{4-n}\big),
		\end{aligned}
	\end{equation}
	where $c$ is some positive constants.
\end{cor}
\begin{proof}
	By the definition of $I[v]$, it follows from Proposition \ref{Prop3.3} that, for $\varepsilon$ sufficiently small, there exists a positive constant $c$ such that
	$$I[v]\geq4c\int_{\Omega}|\Delta v|^2\,\mathrm{d}y.$$
	By substituting this result into \eqref {A3.4}, one has
	$$\begin{aligned}
		S_{\varepsilon V}[u_\varepsilon]\geq S_{\varepsilon V}&[PU_{x,\lambda}]+2c\int_{\Omega}|\Delta v|^2\,\mathrm{d}y+O\bigg(\varepsilon\sqrt{\int_{\Omega}|\Delta v|^2\,\mathrm{d}y}\sqrt{\int_{\Omega}|V|PU_{x,\lambda}^2\,\mathrm{d}y}\bigg)+o\big((\lambda d
		)^{4-n}\big).
	\end{aligned}$$
	Then using the Young inequality, we have
	$$\varepsilon\sqrt{\int_{\Omega}|\Delta v|^2\,\mathrm{d}y}\sqrt{\int_{\Omega}|V|PU_{x,\lambda}^2\,\mathrm{d}y}\leq c\int_{\Omega}|\Delta v|^2\,\mathrm{d}y+\frac{\varepsilon^2}{4c}\int_{\Omega}|V|PU_{x,\lambda}^2\,\mathrm{d}y.$$
	Therefore,
	$$	S_{\varepsilon V}[u_\varepsilon]\geq S_{\varepsilon V}[PU_{x,\lambda}]+c\int_{\Omega}|\Delta v|^2\,\mathrm{d}y+O\bigg(\varepsilon^2\int_{\Omega}|V|PU_{x,\lambda}^2\,\mathrm{d}y\bigg)+o\big((\lambda d
	)^{4-n}\big),$$
	which together with \eqref {A1.10+}, we obtain
	\begin{equation*}
		S_{\varepsilon V}[u_\varepsilon]\geq
		\left\lbrace
		\begin{aligned}
			S_{\varepsilon V}[PU_{x,\lambda}]+c\int_{\Omega}|\Delta v|^2\,\mathrm{d}y+o\big(\varepsilon\lambda^{-4}\log\lambda\big)+o\big((\lambda d\big)^{-4})&\hspace{9mm}\text{if}\hspace{2mm}n=8,\\
			S_{\varepsilon V}[PU_{x,\lambda}]+c\int_{\Omega}|\Delta v|^2\,\mathrm{d}y+o\big(\varepsilon\lambda^{-4}\big)+o\big((\lambda d
			\big)^{4-n})&\hspace{9mm}\text{if}\hspace{2mm}n\geq9.
		\end{aligned}
		\right.
	\end{equation*}
	By the assumption \eqref{A1.1}, we have $\mathcal{S}_2-S_{\varepsilon V}[u_\varepsilon]=(1+o(1))(\mathcal{S}_2-S(\varepsilon V))$. Combining this with \eqref{A1.10+++}, we can complete the proof.
\end{proof}

\section{Proof Of The Main Results}
In this section, we prove Theorems \ref{thm1} and \ref{thm2}. We begin by establishing a key lemma which asserts that the limit $x_0$ (defined in \eqref{dingyi}) lies in $\mathcal{N}(V)$. This lemma plays a crucial role in the proof of our main results.
\begin{lem}\label{lem3.5}
	We have $x_0\in\mathcal{N}(V)$. In particular, $d=O(1) $ and $x\in\mathcal{N}(V)$ as $\varepsilon\rightarrow0^+$.
\end{lem}
\begin{proof}
	We first consider the case $n=8$. By \eqref{A1.39}, we may drop the non-negative term in \eqref{kele} as
	\begin{equation}\label{A3.8}
		\begin{aligned}
			0\geq&\frac{1}{\mathcal{S}_2\lambda^{4}}\big(\mathfrak {a}_8c_8^{4}R(x)+\varepsilon \mathfrak {b}_8c_8^{2}V(x)\log\lambda\big)+o\big(\varepsilon\lambda^{-4}\log\lambda\big)+o\big((\lambda d)^{-4}\big)\\
			=&A(\lambda d)^{-4}-B\varepsilon(\lambda d)^{-4}\log(\lambda d),
		\end{aligned}
	\end{equation}
	where $A=\frac{\mathfrak {a}_8c_8^4d^4R(x)}{\mathcal{S}_2}+o(1)$ and $B=-\frac{\mathfrak {b}_8c_8^{2}d^4\log\lambda}{\mathcal{S}_2\log\lambda d}\big(V(x_0)+o(1)\big)$.
	By \eqref{A3.1+}, it follows that $A$ is positive and bounded away from zero. Hence, from \eqref{A3.8} we can derive $B>0$. Optimizing in $\lambda d$ yields the lower bound
	\begin{equation*}
		A(\lambda d)^{-4}-\varepsilon B(\lambda d)^{-4}\log(\lambda d)\geq-\frac{\varepsilon B}{4e}e^{-\frac{4A}{\varepsilon B}}=-e^{\log\frac{\varepsilon B}{4e}-\frac{4A}{\varepsilon B}}.
	\end{equation*}
	By \eqref{A1.39} again, it follows that there exists $\rho>0$ such that
	\begin{equation}\label{A3.10}
		(1+o(1))\big(\mathcal{S}_2-S(\varepsilon V)\big)\geq e^{-\frac{\rho}{\varepsilon}},
	\end{equation}
	together with \eqref{A3.8}-\eqref{A3.10}, we have
	$$0\geq e^{-\frac{\rho}{\varepsilon}}-e^{\log\frac{\varepsilon B}{4e}-\frac{4A}{\varepsilon B}}.$$
	Therefore,
	\begin{equation}\label{A3.11}
		-\rho\leq\varepsilon\log\frac{\varepsilon B}{4e}-\frac{4A}{B}.
	\end{equation}
	Notice that $B$ is bounded, so the right-hand side of the above inequality can be rewritten as
	\begin{equation}\label{A3.12}
		-\frac{4A}{B}\bigg(1-\frac{\varepsilon B \log B}{4A}\bigg)+\varepsilon\log\frac{\varepsilon}{4e}=-\frac{4A}{B}(1-o(1))+o(1).
	\end{equation}
	By \eqref{A3.11} and \eqref{A3.12}, we obtain $B\geq\frac{A}{\rho}$, which means that $B$ is bounded away from zero. Hence, $d$ is bounded away from zero and $V(x_0)<0$. Moreover, by the continuity of $V$, we can deduce that for $\varepsilon$ small enough, $V(x)<0$, i.e. $x\in\mathcal{N}(V)$. 
	
	Now we treat the case $n\geq 9$ using similar reasoning. In view of \eqref{A1.39} and \eqref{baoshifu}, we get
	\begin{equation}\label{A3.8+}
		\begin{aligned}
			0\geq&\mathcal{S}_2^{1-\frac{n}{4}}\bigg(\frac{\mathfrak {a}_nc_n^{2^\star}R(x)}{{\lambda^{n-4}}}+\varepsilon \mathfrak {b}_nc_n^{2}V(x)\lambda^{-4}\bigg)+o(\varepsilon\lambda^{-4})+o\big((\lambda d)^{4-n}\big)\\
			=&A(\lambda d)^{4-n}-B\varepsilon(\lambda d)^{-4},
		\end{aligned}
	\end{equation}
	where $A=\mathcal{S}_2^{1-\frac{n}{4}}\mathfrak {a}_nc_n^{2^\star}d^{n-4}R(x)+o(1)$ and $B=-\mathcal{S}_2^{1-\frac{n}{4}}\mathfrak {b}_nc_n^2d^4V(x_0)+o(1)$.
	Since $R(x)\gtrsim d^{4-n}$, we know that $A$ is positive and bounded away from zero. Hence, from \eqref{A3.8+} we derive $B>0$. Optimizing in $\lambda d$ yields the lower bound
	\begin{equation*}
		A(\lambda d)^{4-n}-\varepsilon B(\lambda d)^{-4}\geq-\hat{c}A^{-\frac{4}{n-8}}(\varepsilon B)^{\frac{n-4}{n-8}},
	\end{equation*}
	where $\hat{c}:=\frac{n-8}{n-4}\big(\frac{4}{n-4}\big)^{\frac{4}{n-8}}>0$.
	By \eqref{A1.39}, it follows that there exists $\rho>0$ such that 
	\begin{equation}\label{A3.10+}
		(1+o(1))\big(\mathcal{S}_2-S(\varepsilon V)\big)\geq \rho \varepsilon^{\frac{n-4}{n-8}},
	\end{equation}
	which together with \eqref{A3.8+}-\eqref{A3.10+} yields
	\begin{equation*}
		B\geq\rho^{\frac{n-8}{n-4}}\bigg(\frac{n-4}{n-8}\bigg)^{\frac{n-8}{n-4}}\bigg(\frac{A(n-4)}{4}\bigg)^{\frac{4}{n-4}}>0.
	\end{equation*}
	Thus, $B$ is bounded away from zero, which implies that $d$ is bounded away from zero and $V(x_0)<0$. Moreover, the fact that $x\in\mathcal{N}(V)$ for $\varepsilon$ small enough is a consequence of the continuity of $V$. The proof is complete.
\end{proof}

We are now in a position to prove Theorem \ref{thm1}. 
\begin{proof}[Proof of Theorem 1.1]
	We first consider the case $n=8$. By Lemma \ref{lem3.5}, the lower bound \eqref{kele} can be rewritten as (upon dropping the non-negative term)
	\begin{equation*}
		\begin{aligned}
			0&\geq(1+o(1))(\mathcal{S}_2-S(\varepsilon V))+\frac{1}{\mathcal{S}_2\lambda^{4}}\bigg(\mathfrak {a}_8c_8^{4}\big(R(x_0)+o(1)\big)+\varepsilon \mathfrak {b}_8c_8^2\big(V(x_0)+o(1)\big)\log\lambda\bigg)\\
			&\geq(1+o(1))(\mathcal{S}_2-S(\varepsilon V))-\frac{\varepsilon \mathfrak {b}_8c_8^2\big(|V(x_0)+o(1)|\big)}{4e\mathcal{S}_2}e^{-\frac{c_8^2(R(x_0)+o(1))}{10\varepsilon(|V(x_0)+o(1)|)}}.
		\end{aligned}
	\end{equation*}
	Therefore,
	\begin{equation*}
		S(\varepsilon V)\geq \mathcal{S}_2-e^{-\frac{c_8^2 R(x_0)}{10\varepsilon|V(x_0)|}(1+o(1))}\geq \mathcal{S}_2-e^{-\frac{c_8^2 }{10\varepsilon\Phi_8}(1+o(1))},
	\end{equation*}
	where the first inequality follows from \eqref{A1.39} and the second inequality follows from the definition of $\Phi_8$. Since the matching upper bound has already been established in Corollary \ref{cor2.2}, the proof in the case $n=8$ is complete.
	
	Next, we consider the case $n\geq 9$. Similarly, in view of Lemma \ref{lem3.5}, the lower bound \eqref{baoshifu} can be rewritten as
	\begin{equation*}
		\begin{aligned}
			0&\geq(1+o(1))(\mathcal{S}_2-S(\varepsilon V))+\mathcal{S}_2^{1-\frac{n}{4}}\bigg(\frac{\mathfrak {a}_nc_n^{q}\big(R(x_0)+o(1)\big)}{{\lambda^{n-4}}}+\varepsilon \mathfrak {b}_nc_n^2\big(V(x_0)+o(1)\big)\lambda^{-4}\bigg)\\
			&\geq(1+o(1))(\mathcal{S}_2-S(\varepsilon V))-\mathfrak{C}_n\big(R(x_0)+o(1)\big)^{-\frac{4}{n-8}}|V(x_0)+o(1)|^{\frac{n-4}{n-8}}\varepsilon^{\frac{n-4}{n-8}}.
		\end{aligned}
	\end{equation*}
    Therefore,
	\begin{equation*}\begin{aligned}
			S(\varepsilon V)&\geq \mathcal{S}_2-\mathfrak{C}_n\big(R(x_0)+o(1)\big)^{-\frac{4}{n-8}}|V(x_0)+o(1)|^{\frac{n-4}{n-8}}\varepsilon^{\frac{n-4}{n-8}}\\
			&\geq \mathcal{S}_2-\mathfrak{C}_n\Phi_n\varepsilon^{\frac{n-4}{n-8}}+o(\varepsilon^{\frac{n-4}{n-8}}),
		\end{aligned}
	\end{equation*}
	which together with Corollary \ref{cor2.2}, completes the proof in the case $n\geq9$.
\end{proof}

We now prove Theorem \ref{thm2} following the argument in \cite{FKK2020}.
\begin{proof}[Proof of Theorem 1.2]
	We first treat the case $n=8$. By Lemma \ref{lem3.5}, we can rewrite \eqref{kele} as follows
	\begin{equation*}
		0\geq(1+o(1))(\mathcal{S}_2-S(\varepsilon V))
		-\frac{\varepsilon B_\varepsilon}{4e}e^{-\frac{4A_\varepsilon}{\varepsilon B_\varepsilon}}
		+\widetilde{R},
	\end{equation*}	
	where $$\widetilde{R}=\bigg(\frac{A_\varepsilon}{\lambda^4}-\frac{\varepsilon B_\varepsilon\log\lambda}{\lambda^4}+\frac{\varepsilon B_\varepsilon}{4e}e^{-\frac{4A_\varepsilon}{\varepsilon B_\varepsilon}}\bigg)+c\int_{\Omega}|\Delta v|^2\,\mathrm{d}y,$$
	and
	$$A_\varepsilon=\frac{\mathfrak {a}_8c_8^4\big(R(x_0)+o(1)\big)}{\mathcal{S}_2},\quad B_\varepsilon=\frac{\mathfrak {b}_8c_8^2|V(x_0)+o(1)|}{\mathcal{S}_2}.$$
	By virtue of Corollary \ref{cor2.2}, we deduce that
	\begin{equation}\label{A4.8}
		0\geq(1+o(1))e^{-\frac{c_8^2}{10\varepsilon\Phi_8}(1+o(1))}
		-\frac{\varepsilon B_\varepsilon}{4e}e^{-\frac{4A_\varepsilon}{\varepsilon B_\varepsilon}}
		+\widetilde{R}.
	\end{equation}
	Dropping the positive term $\widetilde{R}$ and taking the logarithm on both sides of \eqref{A4.8}, we have
	$$-\frac{4A_\varepsilon}{\varepsilon B_\varepsilon}+\log\frac{\varepsilon B_\varepsilon}{4e}\geq-\frac{c_8^2}{10\varepsilon\Phi_8}(1+o(1))+\log(1+o(1)).$$
	Moreover, multiplying both sides by $\varepsilon$ and then taking the limit, one can see
	$$-\frac{R(x_0)}{|V(x_0)|}\geq-\Phi_8^{-1}.$$
	It follows from the definition of $\Phi_8$ that 
	\begin{equation}\label{A4.9}
		R(x_0)^{-1}|V(x_0)|=\Phi_8.
	\end{equation}
	In view of \eqref{A1.41} and \eqref{A4.9}, we drop the first term on the right hand side of \eqref{A4.8} and retain only the term $c\int_{\Omega}|\Delta v|^2\,\mathrm{d}y$ from $\widetilde{R}$, which yields
	\begin{equation}\label{A4.10}
		\int_{\Omega}|\Delta v|^2\,\mathrm{d}y\leq\frac{\varepsilon B_{\varepsilon}}{4ec}e^{-\frac{c_8^2}{10\varepsilon\Phi_8}(1+o(1))}.
	\end{equation}
	If we drop the last term of $\widetilde{R}$ and multiply both sides of \eqref{A4.8} by $\frac{4e}{\varepsilon B_\varepsilon}e^{\frac{4A_\varepsilon}{\varepsilon B_\varepsilon}}$, we obtain
	$$\begin{aligned}
		1-(1+o(1))\frac{4e}{\varepsilon B_\varepsilon}e^{\frac{4A_\varepsilon}{\varepsilon B_\varepsilon}-\frac{c_8^2}{10\varepsilon\Phi_8}(1+o(1))}&\geq\frac{4e}{\varepsilon B_\varepsilon}e^{\frac{4A_\varepsilon}{\varepsilon B_\varepsilon}}\widetilde{R}\\
		&\geq\frac{4e}{\varepsilon B_\varepsilon}e^{\frac{4A_\varepsilon}{\varepsilon B_\varepsilon}}\bigg(\frac{A_\varepsilon}{\lambda^4}-\frac{\varepsilon B_\varepsilon\log\lambda}{\lambda^4}\bigg)+1\\
		&=1+ye^{y+1}
	\end{aligned}$$
	with $y=\frac{4A_\varepsilon}{\varepsilon B_\varepsilon}-4\log\lambda$. We note that the following equality holds
	$$(1+o(1))\frac{4e}{\varepsilon B_\varepsilon}e^{\frac{4A_\varepsilon}{\varepsilon B_\varepsilon}-\frac{c_8^2}{10\varepsilon\Phi_8}(1+o(1))}=e^{o(\frac{1}{\varepsilon})},$$
    and therefore
	$$-e^{o(\frac{1}{\varepsilon})}\geq ye^{y+1}.$$
	This implies
	$$0<-y\leq o(\frac{1}{\varepsilon}),$$
	which is the same as
	$$\frac{A_\varepsilon}{\varepsilon B_\varepsilon}<\log\lambda\leq\frac{A_\varepsilon}{\varepsilon B_\varepsilon}+ o(\frac{1}{\varepsilon}).$$
	Thanks to \eqref{A4.9}, we deduce that
	\begin{equation}\label{A4.11}
		\lambda=e^{\frac{c_8^2}{40\varepsilon\Phi_8}(1+o(1))}.
	\end{equation}
	Finally, we combine \eqref{A1.1}, \eqref{A1.10++}, \eqref{A3.2}, \eqref{A4.10} and \eqref{A4.11} to obtain
	\begin{equation*}
		|\alpha|^{-2^\star}\mathcal{S}_2^{2}=\mathcal{S}_2^{2}-2^\star\mathfrak {a}_8c_8^{2^\star}R(x_0)\lambda^{-4}+\frac{2^\star(2^\star-1)}{2}\int_{\Omega}U_{x,\lambda}^{2^\star-2}v^2\,\mathrm{d}y+o(\lambda ^{-4}).
	\end{equation*}
	Applying the H\"older inequality and the Sobolev inequality, we get
	\begin{equation}\label{nongfu}
		\int_{\Omega}U_{x,\lambda}^{2^\star-2}v^2\,\mathrm{d}y\lesssim\int_{\Omega}|\Delta v|^2\,\mathrm{d}y,
	\end{equation}
	which together with \eqref{A4.10}-\eqref{nongfu}, yields
	$$|\alpha|=1+e^{-\frac{c_8^2}{10\varepsilon\Phi_8}(1+o(1))}.$$

	Next, we treat the case $n\geq9$. Similarly, using Lemma \ref{lem3.5}, the lower bound \eqref{baoshifu} can be rewritten as $$0\geq(1+o(1))(\mathcal{S}_2-S(\varepsilon V))-\mathfrak{C}_n\big(R(x_0)+o(1)\big)^{-\frac{4}{n-8}}|V(x_0)+o(1)|^{\frac{n-4}{n-8}}\varepsilon^{\frac{n-4}{n-8}}+\widetilde{R},$$
	where $$\widetilde{R}=\bigg(\frac{A_\varepsilon}{\lambda^{n-4}}-\frac{\varepsilon B_\varepsilon}{\lambda^4}+\hat{c}A_\varepsilon^{-\frac{4}{n-8}}(\varepsilon B_\varepsilon)^{\frac{n-4}{n-8}}\bigg)+c\int_{\Omega}|\Delta v|^2\,\mathrm{d}y$$
	and $$A_\varepsilon=\mathcal{S}_2^{1-\frac{n}{4}}\mathfrak {a}_nc_n^{2^\star}\big(R(x_0)+o(1)\big),\quad B_\varepsilon=\mathcal{S}_2^{1-\frac{n}{4}}\mathfrak {b}_nc_n^2|V(x_0)+o(1)|.$$
	Notice that both summands of $\widetilde{R}$ are non-negative. In view of Corollary \ref{cor2.2}, we can see
	\begin{equation}\label{A4.1}
		0\geq \mathfrak{C}_n\bigg(\Phi_n-R(x_0)^{-\frac{4}{n-8}}|V(x_0)|^{\frac{n-4}{n-8}}\bigg)\varepsilon^{\frac{n-4}{n-8}}+\widetilde{R}+o(\varepsilon^{\frac{n-4}{n-8}}),
	\end{equation}
	which implies $R(x_0)^{-\frac{4}{n-8}}|V(x_0)|^{\frac{n-4}{n-8}}\geq\Phi_n$. From the definition of $\Phi_n$, we derive $$R(x_0)^{-\frac{4}{n-8}}|V(x_0)|^{\frac{n-4}{n-8}}=\Phi_n.$$ Furthermore, combining this with \eqref{A4.1}, we have \begin{equation}\label{A4.2}
		\widetilde{R}=o(\varepsilon^{\frac{n-4}{n-8}}).
	\end{equation}
	In particular, \eqref{A4.2} implies that
	\begin{equation}\label{A4.3}
		\int_{\Omega}|\Delta v|^2\,\mathrm{d}y=o(\varepsilon^{\frac{n-4}{n-8}}).
	\end{equation}
	Using Lemma \ref{ALem2} and \eqref{A4.2}, we obtain
	$$\varepsilon^{\frac{n-6}{n-8}}\big(\lambda^{-1}-\lambda_0^{-1}(\varepsilon)\big)^2=o(\varepsilon^{\frac{n-4}{n-8}}),$$
	which implies 
	\begin{equation}\label{A4.4}
		\lambda=\lambda_0+o(\varepsilon^{-\frac{1}{n-8}})=\bigg(\frac{(n-4)\mathfrak {a}_nc_n^{2^\star-2}R(x_0)}{4 \mathfrak {b}_n|V(x_0)|}\bigg)^{\frac{1}{n-8}}\varepsilon^{-\frac{1}{n-8}}+o(\varepsilon^{-\frac{1}{n-8}}),
	\end{equation}
	where $\lambda_0=\lambda_0(\varepsilon)=\big(\frac{(n-4)A_\varepsilon}{4\varepsilon B_\varepsilon}\big)^{\frac{1}{n-8}}$ such that the first summand of $\widetilde{R}$ is equal to zero.
	
	From \eqref{A1.1}, \eqref{A1.10++}, \eqref{A3.2}, \eqref{nongfu}, \eqref{A4.3} and \eqref{A4.4}, it follows that
	\begin{equation}\label{A4.5}
		|\alpha|^{-2^\star}\mathcal{S}_2^{\frac{n}{4}}=\mathcal{S}_2^{\frac{n}{4}}-2^\star\mathfrak {a}_n c_n^{2^\star}R(x_0)\lambda^{4-n}+o(\lambda ^{4-n}).
	\end{equation}
	Hence, combining \eqref{A4.3}-\eqref{A4.5} yields
	$$|\alpha|=1+\mathfrak{D}_n\Phi_n\varepsilon^{\frac{n-4}{n-8}}+o(\varepsilon^{\frac{n-4}{n-8}}).$$
	We complete the proof in the case $n\geq9$.
\end{proof}

\appendix
\section{Auxiliary Results}
In this appendix, we present some detailed computations in our proofs.
\begin{lem}\label{ALem1} Let $x=x_\lambda$ be a sequence of points in $\Omega$ such that $\lambda d(x)\to+\infty$. Then
	\begin{equation*}
		\bigg(\int_{\Omega}|U_{x,\lambda}|^{\frac{2^\star(2^\star-2)}{2^\star-1}}\theta_{x,\lambda}^{\frac{2^\star}{2^\star-1}}\,\mathrm{d}y\bigg)^{\frac{2^\star-1}{2^\star}}=
		\left\lbrace
		\begin{aligned}
			O\big((\lambda d(x))^{4-n}\big)&\hspace{9mm}\text{if}\hspace{2mm}5<n<12,\\
			O\big((\lambda d(x))^{-8}(\log\lambda d(x))^{\frac{2}{3}}\big)&\hspace{9mm}\text{if}\hspace{2mm}n=12,\\
			O\big((\lambda d(x))^{-\frac{4+n}{2}}\big)&\hspace{9mm}\text{if}\hspace{2mm}n>12.
		\end{aligned}
		\right.
	\end{equation*}
	
\end{lem}
\begin{proof}
	In the following, we write $d=d(x)$ for simplicity. It follows from \eqref{A1.13}, \eqref{A1.14} and \eqref{A1.17} that
	$$\begin{aligned}
		\int_{B_{d}(x)}U_{x,\lambda}^{\frac{2^\star(2^\star-2)}{2^\star-1}}\theta_{x,\lambda}^{\frac{2^\star}{2^\star-1}}\,\mathrm{d}y\leq&|\theta_{x,\lambda}|_{L^{\infty}(\Omega)}^{\frac{2^\star}{2^\star-1}}\int_{B_{d}(x)}U_{x,\lambda}^{\frac{2^\star(2^\star-2)}{2^\star-1}}\,\mathrm{d}y\\
		=&O\bigg(\big(d^{4-n}\lambda^{\frac{4-n}{2}}\big)^{\frac{2^\star}{2^\star-1}}\bigg)\int_{B_{d}(x)}U_{x,\lambda}^{\frac{2^\star(2^\star-2)}{2^\star-1}}\,\mathrm{d}y.
	\end{aligned}$$
	Since $\frac{2^\star(2^\star-2)}{2^\star-1}\frac{n-4}{2}=\frac{8n}{n+4}$, it holds that
	\begin{equation*}
		\begin{aligned}
			\int_{B_{d}(x)}U_{x,\lambda}^{\frac{2^\star(2^\star-2)}{2^\star-1}}\,\mathrm{d}y&=O(\lambda^{\frac{8n}{n+4}})\int_0^{d}\frac{r^{n-1}}{(1+\lambda^2r^2)^{\frac{8n}{n+4}}}\,\mathrm{d}r
			=O(\lambda^{\frac{4n-n^2}{n+4}})\int_0^{\lambda d}\frac{t^{n-1}}{(1+t^2)^{\frac{8n}{n+4}}}\,\mathrm{d}t\\
			&=O(\lambda^{\frac{4n-n^2}{n+4}})\bigg(\int_1^{\lambda d}t^{\frac{n^2-12n}{n+4}-1}\,\mathrm{d}t+O(1)\bigg).
		\end{aligned}
	\end{equation*}
    Then we consider three cases separately as follows
	$$\int_1^{\lambda d}t^{\frac{n^2-12n}{n+4}-1}\,\mathrm{d}t=
	\left\lbrace
	\begin{aligned}
		O(1)&\hspace{9mm}\text{if}\hspace{2mm}5<n<12,\\
		O\big(\log(\lambda d)\big)&\hspace{9mm}\text{if}\hspace{2mm}n=12,\\
		O\bigg(\big(\lambda d\big)^{\frac{n^2-12n}{n+4}}\bigg)&\hspace{9mm}\text{if}\hspace{2mm}n>12.
	\end{aligned}
	\right.
	$$
	Having established the estimate in $B_{d}(x)$, we proceed to bound the integral on the complement $\Omega\setminus B_{d}(x)$. Using $|\theta_{x,\lambda}|_{L^{2^\star}(\Omega)}=O(\lambda d)^{\frac{4-n}{2}}$ and the H\"older inequality, we get
	\begin{equation*}
		\begin{aligned}
			\bigg(\int_{\Omega\setminus B_{d}(x)}|U_{x,\lambda}|^{\frac{2^\star(2^\star-2)}{2^\star-1}}\theta_{x,\lambda}^{\frac{2^\star}{2^\star-1}}\,\mathrm{d}y\bigg)^{\frac{2^\star-1}{2^\star}}\leq&|\theta_{x,\lambda}|_{L^{2^\star}(\Omega)}|U_{x,\lambda}|_{L^{2^\star}(\mathbb{R}^n\setminus B_{d})}^{2^\star-2}\\
			=&O\big((\lambda d)^{\frac{4-n}{2}}\big)O\big((\lambda d)^{-4}\big).
		\end{aligned}
	\end{equation*}
	Hence, combining the estimates on $B_{d}(x)$ and $\Omega\setminus B_{d}(x)$, we conclude this proof.
\end{proof}
\begin{lem}\label{ALem2}
	Let $n\geq9$ and $f_\varepsilon:=(0,\infty)\to\mathbb{R}$ be given by
	$$f_\varepsilon(\lambda)=\frac{A_\varepsilon}{\lambda^{n-4}}-\frac{\varepsilon B_\varepsilon}{\lambda^4},$$
	where $A_\varepsilon, B_\varepsilon>0$ are uniformly bounded away from $0$ and $\infty$. Denote by
	$$\lambda_0=\lambda_0(\varepsilon)=\bigg(\frac{(n-4)A_\varepsilon}{4\varepsilon B_\varepsilon}\bigg)^{\frac{1}{n-8}}$$
	the unique global minimum of $f_\varepsilon$. Then there exists a positive constant $c_0$ such that, for any $\varepsilon$, we have 
	\begin{equation*}
		f_\varepsilon(\lambda)-f_\varepsilon(\lambda_0)\geq
		\left\lbrace
		\begin{aligned}
			c_0\varepsilon^{\frac{n-6}{n-8}}\big(\lambda^{-1}-\lambda_0^{-1}(\varepsilon)\big)^2&\hspace{9mm}\text{if}\hspace{2mm}\bigg(\frac{A_\varepsilon}{\varepsilon B_{\varepsilon}}\bigg)^{\frac{1}{n-8}}\lambda^{-1}\leq2\bigg(\frac{4}{n-4}\bigg)^{\frac{1}{n-8}},\\
			c_0\varepsilon^{\frac{n-4}{n-8}}&\hspace{9mm}\text{if}\hspace{2mm}\bigg(\frac{A_\varepsilon}{\varepsilon B_{\varepsilon}}\bigg)^{\frac{1}{n-8}}\lambda^{-1}>2\bigg(\frac{4}{n-4}\bigg)^{\frac{1}{n-8}}.
		\end{aligned}
		\right.
	\end{equation*}
\end{lem}
\begin{proof}
	Let $F(t):=t^{n-4}-t^4$. A straightforward computation shows that $t_0=\big(\frac{4}{n-4}\big)^{\frac{1}{n-8}}$ is the unique global minimum of $F$ on $(0,\infty)$. By the Taylor expansion, there exists $c>0$ such that 
	\begin{equation*}
		F(t)-F(t_0)\geq
		\left\lbrace
		\begin{aligned}
			c\big(t-t_0\big)^2&\hspace{9mm}\text{if}\hspace{2mm}0<t\leq2t_0,\\
			ct_0^{n-4}&\hspace{9mm}\text{if}\hspace{2mm}t>2t_0.
		\end{aligned}
		\right.
	\end{equation*}
	The assertion of the lemma now follows by rescaling. Indeed, it suffices to observe that
	$$f_\varepsilon=A_\varepsilon^{-\frac{4}{n-8}}(\varepsilon B_\varepsilon)^{\frac{n-4}{n-8}}F\bigg(\bigg(\frac{A_\varepsilon}{\varepsilon B_\varepsilon}\bigg)^{\frac{1}{n-8}}\lambda^{-1}\bigg)$$
	and to use the boundedness of $A_\varepsilon$ and $B_\varepsilon$.
\end{proof}
	 
	 \section*{Statements and Declarations}
	 
	 \textbf{Data Availability} Data sharing not applicable to this article as no datasets were generated or analysed during the current study.
	 
	 \textbf{Competing Interests} The authors have no competing interests to declare that are relevant to the content of this article.
	 
	 \textbf{Author Contribution declaration} The authors declare that they contribute to the research of this paper equally.
	 
	 \textbf{Funding declaration} Minbo Yang was partially supported by National Key Research and Development Program of China (No. 2022YFA1005700), National Natural Science Foundation of China (12471114) and the Natural Science Foundation of Zhejiang Province (LZ26A010002).
	\bibliographystyle{plain}

\end{document}